\documentclass[11pt]{article}

\usepackage{amsfonts,amsthm,amssymb,amsmath}
\usepackage{graphicx,multirow}
\usepackage{dsfont,bm} 
\usepackage[usenames,dvipsnames]{color}

\usepackage[top=1in,bottom=1in,left=1in,right=1in, a4paper]{geometry}
\usepackage[round]{natbib}
\usepackage{authblk}

\usepackage[pdftex]{hyperref}
\hypersetup{%
pdftitle={Insensitivity of Proportional Fairness in Critically Loaded Bandwidth Sharing Networks},%
pdfauthor={Maria Vlasiou, Jiheng Zhang, Bert Zwart},%
bookmarks=true,%
bookmarksnumbered=true,%
pdfstartview={FitH},%
colorlinks=false,%
breaklinks=true,%
}%

\numberwithin{equation}{section}
\numberwithin{table}{section}
\numberwithin{figure}{section}

\newtheorem{definition}{Definition}[section]
\newtheorem{theorem}{Theorem}[section]
\newtheorem{lemma}{Lemma}[section]
\newtheorem{proposition}{Proposition}[section]

\newcommand{\R}{\mathbb{R}} 

\newcommand{\D}{\mathcal{D}} 

\newcommand{\pr}{\mathbb{P}} 
\newcommand{\E}{\mathbb{E}} 

\newcommand{\dto}{\Rightarrow} 
\newcommand{\uoc}{\emph{u.o.c.\ }} 
\newcommand{\osc}[2]{\textrm{Osc}({#1},{#2})} 

\newcommand{\fl}[1]{\lfloor{#1}\rfloor} 

\newcommand{\dist}{\mathbf{d}^{fp}} 
\newcommand{\id}[1]{\mathds{1}_{\{{#1}\}}} 
\newcommand{\invw}{\mathcal{W}}

\newcommand{\route}{\mathcal{R}} 
\newcommand{\link}{\mathcal{L}} 
\newcommand{\phase}{\mathcal{F}} 

\newcommand{\ext}[1]{\bm{#1}} 
\newcommand{\vect}[1]{\bm{#1}} 
\newcommand{\diag}[1]{\mathrm{diag}\left({#1}\right)} 

\newcommand{\fs}[1]{\bar{#1}^k} 
\newcommand{\sfs}[2]{\bar{#2}^{k,#1}} 
\newcommand{\fss}[1]{\tilde{#1}^k} 
\newcommand{\ds}[1]{\hat{#1}^k} 

\usepackage[normalem]{ulem}

\title{
Insensitivity of Proportional Fairness in Critically Loaded Bandwidth Sharing Networks
}

\author[*]{Maria Vlasiou}
\author[$\dagger$]{Jiheng Zhang}
\author[$\ddag$]{Bert Zwart}
\affil[*]{Eindhoven University of Technology}
\affil[$\dagger$]{The Hong Kong University of Science and Technology}
\affil[$\ddag$]{Centrum Wiskunde \& Informatica, Amsterdam}

\begin{document}

\maketitle

\begin{abstract}
\noindent
Proportional fairness is a popular service allocation mechanism to describe and analyze the performance of data networks at flow level.
Recently, several authors have shown that the invariant distribution of such networks admits a product form distribution under critical loading.
Assuming exponential job size distributions, they leave the case of general job size distributions as an open question.
In this paper we show the conjecture holds for a dense class of distributions. This yields a key example of a stochastic network in which the heavy traffic limit has an invariant
distribution that does not depend on second moments. Our analysis relies on a uniform convergence result for a fluid model which may be of independent interest.
\end{abstract}



\section{Introduction}
\label{sec:introduction}

A popular way to model congestion of data traffic is to consider such traffic at a level where files or jobs are represented by continuous flows, rather than discrete packets.
This gives rise to bandwidth sharing networks, as introduced in \cite{MR99b}.
Such networks model the dynamic interaction among flows that compete for bandwidth along their source-destination paths.
Apart from offering insight into the complex behavior of computer-communication networks, they have also recently been suggested to analyze road-traffic congestion (see for instance \cite{KellyWilliams2010}).
The analysis of bandwidth sharing networks is challenging,  requiring tools from both optimization and stochastics.

Perhaps the most important bandwidth allocation mechanism that has been considered so far is {\em proportional fairness}.
In a static setting, this policy can be implemented in a distributed fashion, simultaneously maximizing users' utility, cf.\ \cite{Kelly1997,YiChiang2008}.
In addition, proportional fairness is known to be the only policy that satisfies the four axioms of Nash bargaining theory (\cite{Mazumdar1991,StefanescuStefanescu1984}).
These are desirable properties in a static setting. Furthermore, proportional fairness  has  attractive dynamic properties:
while being a greedy policy, proportional fairness has also shown to optimize some long term cost objectives, at least in a heavy traffic environment (\cite{YeYao2012}).
In particular, it is known to be stable under natural traffic conditions in internet flow-level models (\cite{Massoulie2007}).
Recently, proportional fairness has been suggested as an attractive alternative to maximum pressure policies in \cite{Walton2014a}.

In some special cases detailed below, a bandwidth sharing network operating under proportional fairness admits an invariant distribution for the number of users which is computable. As these cases are rather restrictive, it is natural to obtain insight in the performance of proportional fairness for more general network topologies. In \cite{KKLW2009}, it is shown, assuming exponential job size distributions, that the performance of proportional fairness is still tractable if  the network is heavily loaded. Under a heavy traffic assumption, a limit theorem is developed yielding an approximating semimartingale reflected Brownian motion (SRBM), of which the invariant distribution is shown to have a product form. A restrictive assumption in \cite{KKLW2009} (the so-called `local traffic assumption' stating that each link in the network serves a route consisting only of that link) was removed in \cite{YeYao2012} by using elegant geometric arguments. While \cite{YeYao2012} allow for generally distributed flow sizes, they do so assuming that the service policy within a class is first-in-first-out (FIFO), which is well-suited for packet level models \cite{Walton2014Allerton}. In the present paper, we focus on flow level models, in which the per-class discipline is Processor Sharing (PS); this discipline is harder to analyze than FIFO and corresponds to the original open question posed in \cite{KKLW2009}. A recent survey on these developments can be found in \cite{Williams2015}.

While the Poisson arrival assumption can often be justified to some degree in practice, the same cannot be said for exponential job size distributions.
As such, it is desirable for the performance of a network to be insensitive to fluctuations in higher moments of the job size distribution.  There is overwhelming statistical evidence that the variance of file sizes is in fact infinite (\cite{Resnick1997}), which can have dramatic impact on performance (\cite{ZwartBorstMandjes2005}). As perfectly stated in \cite{BonaldProutiere2003}: ``the practical value of insensitivity is best illustrated by the enduring success of Erlang's loss formula in telephone networks''. In \cite{BonaldProutiere2003}, it is shown that proportional fairness is the only utility maximizing policy that yields this insensitivity property, provided the network topology has a hypercube structure and that all servers work at the same speed. Given these limitations on the insensitivity of proportional fairness, some related allocation mechanisms have been suggested  that yield insensitivity for arbitrary networks topologies. One such suggestion is balanced fairness (\cite{BonaldProutiere2003}), based on connections with Whittle networks. Another suggestion (\cite{Massoulie2007}) is modified proportional fairness. However, neither of these two policies are utility maximizing.

Though proportional fairness itself may not be always insensitive, it remains a key allocation mechanism for the reasons mentioned above. In fact, the key question addressed but left open in both \cite{KKLW2009} and \cite{YeYao2012}, is whether the product form property of their heavy traffic approximation, derived for exponential job sizes, would still hold for more general job size distributions, yielding insensitivity of proportional fairness in heavy traffic.

The goal of this paper is to provide an affirmative answer to this question, providing both a new perspective of insensitivity in bandwidth sharing networks, as well as establishing new heavy-traffic limits. Postponing a formal description to later sections, we give an informal explanation of our main result. We show that the vector $N$ of the number of users along each route in steady state can be approximated as follows:
\begin{equation}
  \label{mainapproximation}
  N \approx \diag{\rho} A^T E_s.
\end{equation}
Here $\diag{\rho}$ is a diagonal matrix having the load of each route on the diagonal. $A$ is a 0-1 matrix encoding which server (link) is used by which route, and $E_s$ is a vector of independent exponential random variables. Each random variable corresponds to a server, and has as parameter the slack of that resource, i.e.\ if $c$ is the vector of service speeds, then $s = c-A\rho$. The random variables $E_s$ can actually be interpreted as equilibrium values of the Lagrange multipliers associated with the resources.
In \cite{Walton2014Allerton} this property is called {\it product form resource pooling}.
We should note upfront that \eqref{mainapproximation} is based on the steady-state of our heavy traffic limit; we do not interchange heavy traffic and steady state limits. In the case of exponential job sizes, this interchange is established in \cite{STZ2014}.
\cite{JonckheereLopez2014} establish insensitivity of large deviation rate functions assuming the network has a tree topology. Other recent developments of proportional fairness are described in \cite{HMSY2014}.

Our result \eqref{mainapproximation} relies on the assumption that a link in the network is work-conserving. When individual users have additional constraints on their individual access rates, \eqref{mainapproximation} no longer holds, and the distribution of $N$ is better approximated by a multivariate normal, cf.\ \cite{ReedZwart2013}. When relaxing the assumption of proportional fairness to other utility maximizing bandwidth allocation policies, the theory becomes much harder and is still partly conjectural, as the resulting SRBM's no longer live in polyhedral domains, cf.\ \cite{KangWilliams2007,KKLW2009}. In this case, the simple approximation \eqref{mainapproximation} cannot be expected to hold.
Another assumption is that $A$ is of full row rank. \cite{KMW2009} show that  \eqref{mainapproximation} may not hold in in general if $A$ is not of full row rank.
Extensions to multi-path routing, of which its nature and  importance is described in \cite{KKLW2009}, require the elements of $A$ to be nonnegative rather than 0-1, which is not a restriction for the analysis in our paper.

In our analysis, we additionally assume that job size distributions have a particular phase-type structure, which is non-restrictive in the sense that any distribution with non-negative support can be approximated arbitrary closely by such a phase-type distribution. This assumption is technically convenient as it allows for a finite-dimensional Markovian description of the system. Extending our results to more general distributions requires a measure-valued state descriptor, and is beyond the scope of the techniques developed in this paper. Note that this  would still not cover the practically relevant case of job sizes with infinite variance, which has not even been resolved even in the single-node single-class case, cf.\ \cite{LSZ2013}. In the present paper, second moments show up in the description of the process limit, but cancel out against one another while computing the invariant distribution of the SRBM, using the skew symmetric condition developed by \cite{HarrisonWilliams1987}. In particular, we show that the covariance matrix of our SRBM is twice the reflection matrix.

Our justification of \eqref{mainapproximation} is based on the main technical results of this paper, which are Theorems \ref{thm:fluid-conv-inv}, \ref{thm:diffusion} and \ref{thm:productform} below. To derive these results, we adapt the state-space collapse approach of \cite{Bramson1998,Williams1998,Stolyar2004} to our setting, building also on  \cite{Bramson1996b,KKLW2009,Massoulie2007,YeYao2012}. Specifically, we first investigate a fluid model assuming the system is critically loaded, and define a critical fluid model extending \cite{Massoulie2007}. Adapting techniques from \cite{YeYao2012} and \cite{KKLW2009}, we characterize and investigate the set of invariant points of the fluid model.

We then proceed with the main technical challenge of this paper, which is to show that fluid model solutions converge uniformly to an invariant point, at an exponential rate, which is Theorem \ref{thm:fluid-conv-inv}. Ideas from \cite{Bramson1996b} and \cite{Massoulie2007} form a useful starting point, but the analysis pertaining to our setting demands significant additional work. Our main idea is the analysis of a candidate Lyapunov function through a novel application of a rearrangement inequality, significantly simplifying \cite{Massoulie2007}. The resulting upper bound on the derivative of this function is then bounded further using properties like the utility-maximizing nature of proportional fairness. The fact that the proportionally fair bandwidth allocation function may be discontinuous at the boundary complicates the analysis. The analysis of the fluid model is not restricted to phase type routing. Instead, we consider general Markovian routing, expecting the convergence result to be useful beyond its present application, though we need to assume that all external arrival rates are positive.
With the uniform convergence of fluid model solutions in place, the remaining steps follow arguments similar to \cite{YeYao2012}, using in particular some of their intermediate results. This yields the diffusion limit in Theorem \ref{thm:diffusion} and its invariant distribution in Theorem \ref{thm:productform}.

The paper is organized as follows. The network model, and some assumptions are introduced in Section \ref{sec:network-model}. In Section \ref{sec:system-dynamics}, we give a detailed description of the dynamics of our model. These dynamics are rewritten in Section \ref{sec:geometry-fixed-point}, and interpreted in terms of what we expect to see in heavy traffic. An auxiliary fluid model with general Markovian routing is introduced and analyzed in detail in Section \ref{sec:fluid-model}. This paves the way to obtain the diffusion limit in Section \ref{sec:diff-appr}, of which the invariant distribution is computed in Section \ref{sec:invariant}.

\section{The network model}
\label{sec:network-model}

In this section, we provide a detailed model description. As we make heavy use of results from \cite{YeYao2012}, we follow their notation whenever possible. All vectors are column vectors. Throughout the paper, $e$ is a column vector with all elements equal to $1$ and $I$ denotes the identity matrix.  The dimensions of $e$ and $I$ should be clear from the context.

\paragraph{Network structure.}
The network consists of a set of routes $\route=\{1,\ldots, R\}$, which are typically indexed by $r$. Each route traverses several links, which are indexed by $l$, $l \in \link=\{1, \ldots, L\}$. Each link has a service capacity $c_l$. Let $A$ denote the \emph{link-route} matrix of dimension $L\times R$. $A_{l,r}=1$ if route $r$ needs 1 unit of capacity from link $l$ and $0$ otherwise. Assume $A$ has full row rank; hence $L\le R$; we note that all arguments in the paper remain valid if $A$ is a nonnegative matrix of full row rank.

\paragraph{Stochastic assumptions.}
Next, we introduce the arrival process and service time assumptions. We assume for convenience that arrival processes are Poisson with rate $\lambda_r$. Service times at route $r$ follow a phase type distribution with $F_r$ phases. The set $\phase_r=\{1,\ldots, F_r\}$ contains all phases for jobs on route $r$. As is commonplace (cf.\ \cite{Asmussen2003}), a phase-type random variable is the lifetime of an absorbing Markov chain with initial distribution $\vect{a}_r=(\ext{a}_{r,1},\ldots, \ext{a}_{r,F_r})^T\in \R^{F_r}_+$, sub-stochastic transition matrix $P^r\in \R^{F_r}\times\R^{F_r}$, and rates $\vect{\mu}_r=(\ext{\mu}_{r,1},\ldots, \ext{\mu}_{r,F_r})^T\in \R^{F_r}_+$; i.e.\ the service time in phase $f$ is exponentially distributed with rate $\ext{\mu}_{r,f}$. In particular, the mean service time at phase $f$ on route $r$ is $\ext{m}_{r,f} = \frac{1}{\ext{\mu}_{r,f}}$, and
$\vect{m}_r=(\ext{m}_{r,1},\ldots, \ext{m}_{r,F_r})^T\in \R^{F_r}_+$.
We assume
  \begin{align}
    \label{eq:cond-phase-route-init}
    & \lambda_r\ext{a}_{r,f}>0 \textrm{ for all } f\in\phase_r \textrm{ and } r\in\route,\\
    \label{eq:P-phase-inv}
    & (I-P^r) \textrm{ is invertible}.
  \end{align}
The first assumption, that all routes have arrivals for each phase, is non-standard, and required in our analysis in Section \ref{sec:fluid-model}. It is non-restrictive in the sense that an inspection of the proof of
 \cite[Theorem~III.4.2]{Asmussen2003} shows that the resulting class of distributions is still dense in the class of all distributions with non-negative support.
Let $P^{r,T}$, $\vect{a}_r^T$ and $\vect{m}_r^T$ denote the transpose of $P^r$, $\vect{a}_r$ and $\vect{m}_r$, then the mean service requirement $\beta_r$ at route $r$ is
  \begin{equation}
    \label{eq:mean-MixErlang}
    \beta_r = \vect{m}_r^T (I-P^{r,T})^{-1} \vect{a}_r.
  \end{equation}

\paragraph{State-space description.} Denote the $R$-dimensional vector of jobs on each route by $n=(n_1,\ldots,n_R)^T$ with $n_r$ being the number of jobs on route $r\in\route$. To obtain a Markovian description of our network, it is useful to introduce a more detailed state space descriptor
\begin{equation}
  \label{eq:ext-colum-interpre}
  \ext n =
  \left(
    \ext n_{1,1},\ldots, \ext n_{1,F_1},
    \ \ldots\ldots\ ,
    \ext n_{R,1}, \ldots , \ext n_{R,F_R}
  \right)^T,
\end{equation}
with $\ext n_{r,f}$ denoting the number of jobs on phase $f$ at route $r$.
It is clear that $\ext n$ is a  $\sum_{r\in\route}F_r$-dimensional vector and $n_r=\sum_{f\in\phase_r} \ext n_{r,f}$.
We also need a \emph{link-phase} matrix, denoted by $\ext A$ which is of dimension $L\times \sum_{r\in\route}F_r$.
\begin{equation}
  \label{eq:link-phase-matrix}
  \ext A_{l,f} = A_{l,r} \ \textrm{ for all } f=\sum_{r'=1}^{r-1}F_{r'}+1,\ldots, \sum_{r'=1}^{r}F_{r'}.
\end{equation}
Thus, $\ext A$ is obtained by taking the $r$th column of $A$ and repeating it for $F_r$ times.
From now on, when we make a distinction between routes and phases, we speak of `route level' and `phase level'.
The associated notation will be distinguished by using boldface.

\paragraph{Traffic load.} The route-level traffic load for each $r\in\route$ is
\begin{equation}
  \label{eq:rho}
  \rho_r = \lambda_r\beta_r.
\end{equation}
Denote $\rho=(\rho_1,\ldots,\rho_R)^T\in\R_+^R$ and $c=(c_1,\ldots,c_L)^T\in\R_+^L$, then
\begin{equation}
  \label{eq:multi-bottleneck}
  A\rho = c.
\end{equation}
A link $l$ is said to be a bottleneck if $A_l \rho = c_l$.  For convenience, we assume that all links are a bottleneck. This assumption can be removed along the lines of the electronic companion of \cite{YeYao2012}. Note however that we assume \eqref{eq:multi-bottleneck} for our limiting process. Later on, we introduce a sequence of processes, indexed by $k$, for which $c - A\rho^{(k)}$ is of the order $1/k$.

Let $\diag{x}$ be a diagonal matrix that contains the element of a vector $x$. The traffic load for each route $r$ at the phase level is defined as
\begin{equation}
  \label{eq:rho-phase}
  \vect{\rho}_r = \lambda_r[\diag{\vect{m}_r} (I-P^{r,T})^{-1}\vect{a}_r].
\end{equation}
In other words, $\vect{\rho}_r=(\ext{\rho}_{r,1},\ldots,\ext{\rho}_{r,F_r})^T\in \R^{F_r}_+$.
It is clear from \eqref{eq:mean-MixErlang} and \eqref{eq:rho} that the aggregated load for each route $r$ is
\begin{equation}
  \label{eq:rho-phase-route}
  \rho_r = \sum_{f\in\phase_r}\ext\rho_{r,f}.
\end{equation}

\paragraph{Proportional fairness allocation.}
Denote by $\Lambda_r(n)$, $r\in\route$, the capacity allocated to route $r$ jobs when the network status is $n$. Let $\Gamma$ denote the set of all feasible allocations, i.e.\
\begin{equation}
  \label{eq:feasible-set-route}
  \Gamma = \left\{\gamma\in\R^R: A\gamma\le c, \gamma\ge 0\right\}.
\end{equation}
The proportional fair allocation $\Lambda(n)$ is the solution to the optimization problem
\begin{equation}
  \label{eq:opt-route}
  \max_{\gamma\in\Gamma} \sum_{r\in\route} n_r \log(\gamma_r), 
\end{equation}
with $\Lambda_r(n)=0$ if $n_r=0$.
According to the optimality condition, any optimal solution to \eqref{eq:opt-route} satisfies
\begin{equation}
  \label{eq:kkt}
  \frac{n_r}{\gamma_r} = \sum_{l\in\link}A_{l,r}\eta_l, \quad r\in\route,
\end{equation}
for some $\eta = (\eta_l)_l \in\R_+^L$.
It is known that $\Lambda$ is directionally differentiable on $(0,\infty)^R$ by \cite{ReedZwart2013} (earlier \cite{KellyWilliams2004} established continuity). In addition, $\Lambda$ is radially homogeneous, i.e.\ $\Lambda (yn)=\Lambda (n)$ for $y>0$ \cite{KellyWilliams2004}.

The allocation to each phase $f$ on route $r$ is $\ext\Lambda_{r,f}(\ext n)=\frac{\ext n_{r,f}}{n_r}\Lambda_r(n)$, where we make the convention throughout the paper that $0/0=0\times \infty = 0$.
This is consistent with the fact that $\ext\Lambda(\ext n)$, as a $\sum_{r\in\route}F_r$-dimensional vector, is the optimal solution to
\begin{equation}
  \label{eq:opt-phase}
  \max_{\gamma\in\ext \Gamma} \sum_{r\in\route, f\in\phase_r} \ext n_{r,f} \log(\ext \gamma_{r,f}),
\end{equation}
where
\begin{equation*}
  \ext \Gamma = \left\{
    \ext\gamma\in \R^{\sum_{r\in\route}F_r}:
    \ext A \ext\gamma \le c, \ext\gamma\ge 0
  \right\}.
\end{equation*}
The extended vector $\ext \gamma$, together with $\ext \mu$, $\ext m$ and $\ext\rho$, is interpreted in the same way as \eqref{eq:ext-colum-interpre}.

\section{System dynamics}
\label{sec:system-dynamics}


Let $\ext N^k_{r,f}(t)$ denote the number of jobs on route $r$ at phase $f$; $N^k_r(t)=\sum_{f\in\phase_r}\ext N^k_{r,f}(t)$ denotes the total number of jobs on route $r$. Set the column vector $\ext N^k(t)= (\ext N^k_{r,f}(t))$.
The resource allocated to phase $f$ on route $r$ at time $t$ is $\Lambda_{r,f}(\ext N^k(t))$ according to \eqref{eq:opt-phase}.

For convenience, set $P^{r}_{f,0}=1-\sum_{f'\in\phase_r}P^r_{f,f'}$.
Let $\ext E_{r,f}$, $\ext S_{r,f,f'}$, $r\in\route$, $f\in\phase_r$, $f'\in\phase_r \cup \{0\}$, denote independent unit rate Poisson processes.
The dynamics of $\ext N^k(t)$ can be written as
\begin{equation}
  \label{eq:dynamics-k}
  \begin{split}
    \ext N^k_{r,f}(t) &= \ext N^k_{r,f}(0) + \ext E_{r,f}(\lambda_r^k \ext{a}_{r,f}t)
    + \sum_{f'\in\phase_r}\ext S_{r,f',f}\Big(\ext{\mu}_{r,f'} P^r_{f',f} \ext D_{r,f'}(t)\Big)\\
    &\quad - \sum_{f'\in\phase_r \cup \{0\}}\ext S_{r,f,f'}\Big(\ext{\mu}_{r,f} P^r_{f,f'} \ext D^k_{r,f}(t)\Big),
   \end{split}
\end{equation}
where
\begin{equation}
  \label{eq:accum-allocation}
  \ext D^k_{r,f}(t) = \int_0^t\ext \Lambda_{r,f}(\ext N^k(s))ds.
\end{equation}
As users at a given route and phase may not leave the network immediately we define a phase-based workload $\ext W^k(t)$ as a  $\sum_{r\in\route}F_r$-dimensional vector interpreted  as in \eqref{eq:ext-colum-interpre}. In particular, setting
\begin{equation}
  \label{eq:P-matrix}
  \ext P=
  \begin{pmatrix}
    P^1 & 0 & 0 & 0 \\
    0 & P^2 & 0 & 0 \\
    \vdots  & \vdots  & \ddots & \vdots  \\
    0 & 0 & \cdots & P^R
  \end{pmatrix},
\end{equation}
the phase-base workload is now defined as
\begin{equation}
  \label{eq:workload-phase}
  \ext W^k(t) = \diag{\ext m} (I-\ext P^T)^{-1} \ext N^k(t),
\end{equation}
This is not the true workload, but is a convenient proxy and a custom choice in heavy traffic analysis; see \cite{Harrison2000} for background.

Define the centered processes
\begin{align}
  \label{eq:center-E}
  \breve{\ext E}^k_{r,f}(t) &= \ext E_{r,f}(\lambda_r^k \ext{a}_{r,f} t)-\lambda_r^k \ext{a}_{r,f}t,\\
  \label{eq:center-S-r-ff'}
  \breve{\ext S}^k_{r,f,f'}(t) &= \ext S^k_{r,f,f'}(\ext{\mu}_{r,f}P^r_{f,f'} t)-\ext{\mu}_{r,f}P^r_{f,f'}t,
\end{align}
and
\begin{align}
  \label{eq:center-S-+}
  \left(\breve{\ext S}^k_+(\ext D^k(t))\right)_{r,f}
  &=\sum_{f'\in\phase_r}\breve{\ext S}^k_{r,f',f}(D^k_{r,f'}(t)),\\
  \label{eq:center-S--}
  \left(\breve{\ext S}^k_-(\ext D^k(t))\right)_{r,f}
  &=\sum_{f'\in\phase_r \cup \{0\}}\breve{\ext S}^k_{r,f,f'}(D^k_{r,f}(t)).
\end{align}
Define the vector $\ext \rho^k$  as in \eqref{eq:rho-phase} with $\lambda_r$ replaced by $\lambda^k_r$. It follows from \eqref{eq:dynamics-k} and \eqref{eq:accum-allocation} that
\begin{equation}
  \label{eq:dynamics-k-workload}
  \ext W^k(t) = \ext W^k(0) + \ext X^k(t) + \int_0^t [\ext\rho - \ext \Lambda(\ext N^k(s))] ds
\end{equation}
where
\begin{equation}
  \label{eq:centered-k}
  \ext X^k(t) = (\ext\rho^k - \ext\rho) t
 + \diag{\ext m} (I-\ext P^T)^{-1} \left[\breve{\ext E}^k(t)+ \breve{\ext S}^k_+(\ext D^k(t)) - \breve{\ext S}^k_-(\ext D^k(t))\right],
\end{equation}
Let
\begin{equation}
  \label{eq:Y-k}
  Y^k(t) = \ext A\int_0^t [\ext \rho - \ext \Lambda (\ext N^k(s)] ds.
\end{equation}
It is easily seen that
\begin{align*}
  Y^k(t) &= \int_0^t [c-\ext A \ext \Lambda(\ext N^k(s))]ds
  = \int_0^t [c- A \Lambda(N^k(s))]ds.
\end{align*}
The vector $Y^k(t)$ can be interpreted as service capacities that have not been used in $[0,t]$. 

\section{Geometry of the fixed-point state space}
\label{sec:geometry-fixed-point}

In this section we determine and analyze the set of points $\ext n$ for which $\ext \Lambda (\ext{n}) = \ext \rho$.
A main technical task is to show that only such states $n$ show up in the heavy-traffic limit. The analysis in this section is inspired by \cite[Section 3 ]{YeYao2012}, though our situation is  different, as we need to deal with routing. Since we only consider local routing, it is possible to utilize their results.

Let $\ext B^\dagger= \diag{\ext m}(I-\ext P^T)^{-1}\diag{\ext \rho}$. Define
\begin{equation}
  \label{eq:inv-W}
  \invw := \{w=\ext B^\dagger\ext A^T\pi:\pi=(\pi_l)_{l\in\link}\ge 0\}.
\end{equation}
Observe that $\ext B^\dagger$ is a block-diagonal matrix. Let $\ext C$ be an $R \times \sum_{r\in\route} {F}_r$ matrix with the first $F_1$ columns all being an $R$-dimensional vector $(1,0,\ldots,0)^T$, and the next $F_2$ columns all being $(0,1,\ldots,0)^T$ and so on. For example,
\begin{equation*}
  \ext C =
  \begin{pmatrix}
    1 & \ldots & 1 & 0 & \ldots & 0 & \ldots & 0 & \ldots & 0\\
    0 & \ldots & 0 & 1 & \ldots & 1 & \ldots & 0 & \ldots & 0\\
    \vdots  & \vdots  & \vdots & \vdots  & \vdots  & \vdots & \vdots & \vdots  & \vdots & \vdots\\
    0 & \ldots & 0 & 0 & \ldots & 0 & \ldots & 1 & \ldots & 1
  \end{pmatrix}.
\end{equation*}
Then we have
\begin{equation}
  \label{eq:A-ext-A}
  \ext A = A \ext C.
\end{equation}
Define now the diagonal matrix
\begin{equation}
  \label{eq:B}
  \ext B:= \diag{\ext m}\diag{(I-\ext P^T)^{-1} \ext \rho}.
\end{equation}
Due to the  structure of $\ext A$ (repeating the $r$th column of $A$ for $F_r$ times, see \eqref{eq:link-phase-matrix}), we have
\begin{equation}
  \label{eq:ext-B-equiv}
  \ext B^\dagger \ext A^T=\ext B \ext A^T.
\end{equation}
Due to \eqref{eq:kkt}, we have
\begin{equation*}
  \frac{\ext n_{r,f}}{\ext\gamma_{r,f}} = \sum_{l\in\link}A_{l,r}\eta_l, \quad r\in\route.
\end{equation*}
To connect with the initial motivation of the section, we elaborate on how $\invw$ arises. Suppose $\ext\gamma_{r,f}=\ext\rho_{r,f}$ is the optimal solution to \eqref{eq:opt-phase} and let the $\pi_l=\eta_l$ be the corresponding shadow price. According to \eqref{eq:workload-phase}, the workload of phase $f$ on route $r$ is
\begin{align*}
  \ext w_{r,f} = \ext{m}_{r,f}\sum_{f'}[(1-P^{r,T})^{-1}]_{f,f'}\ext n_{r,f'}
  = \ext{m}_{r,f}\sum_{f'}[(1-P^{r,T})^{-1}]_{f,f'}\ext\rho_{r,f'}\sum_{l\in\link}A_{l,r}\pi_l.
\end{align*}
In matrix form,  $\ext w = \ext B \ext A^T \pi$. $\invw$ contains all states $\ext n$ with $\ext\Lambda_{r,f}(\ext n)=\ext \rho_{r,f}$ for $\ext n>0$. Therefore, $\invw $ is the so-called \emph{invariant manifold}, or \emph{fixed-point state space} associated with the workload process $W^k$ defined in the previous section.

The key difference between our model and that of \cite{YeYao2012} is in the definition of the workload in \eqref{eq:workload-phase}. As a consequence, the matrix $\ext B^\dagger$ is not a diagonal matrix, as required in the geometric analysis in \cite{YeYao2012}. However, due to the special structure of local routing \eqref{eq:P-matrix}, we can replace $\ext B^\dagger$ with $\ext B$ (cf.\ \eqref{eq:ext-B-equiv}) and the structure of $\invw$ coincides with that of the similar manifold introduced in \cite{YeYao2012}. Thus, all the analysis  in \cite{YeYao2012} applies to our situation; this would no longer be the case if we consider full Markovian routing. We now briefly cite some relevant results from \cite{YeYao2012}.

\paragraph{Workload decomposition.}
Let $\Delta$ be the left null space of $\ext A^T$, i.e.\ the kernel of $A$:
\begin{equation*}
  \Delta : = \{ \delta \in \R^{\sum_{r\in\route}F_r}: \ext A \delta = 0\}.
\end{equation*}
as $A$ is of full row rank. We assume without loss of generality that ${\sum_{r\in\route}F_r}> L$; if equality would hold, then this would actually simplify the analysis, as $\invw$ is the positive
orthant in this case.
$\Delta$ is of dimension ${\sum_{r\in\route}F_r}-L$. Since $\ext B$ is diagonal, and thus of full rank, then for any base $H$ (which is of dimension ${\sum_{r\in\route}F_r}\times ({\sum_{r\in\route}F_r}-L)$) of $\Delta$, $\ext B H$ is also a base and
\begin{equation}
  \label{eq:H-matrix}
  \ext A \ext B H=0.
\end{equation}
Moreover, as $\ext B$ is symmetric, one can chose the base $H$ such that
\begin{equation*}
   H^T \ext B H=I.
\end{equation*}
The null space $\Delta$ can now be expressed as
\begin{equation}
  \label{eq:null-space}
  \Delta = \{ \ext B H z: z\in\R^{{\sum_{r\in\route}F_r}-L} \}.
\end{equation}
So any ${\sum_{r\in\route}F_r}$ dimensional real-valued vector $w$ can be decomposed into two linearly independent vectors, one belonging to $\invw$ and one belonging to ${\Delta}$:
\begin{equation}
  \label{eq:w-decompose-AH}
  w = \ext B \ext A^T\pi + \ext B H z,
\end{equation}
with $\pi$ and $z$ as specified in \eqref{eq:inv-W} and \eqref{eq:null-space}. Note that because $\ext A$ and $\ext B$ are both full rank and $\ext A\ext B$ is surjective, $\ext A \ext B \ext A^T$ is invertible. Then set $G= \ext A^T(\ext A \ext B \ext A^T)^{-1}$ and observe that
\begin{equation}
  \label{eq:G-matrix}
   G^T \ext B^T G = G^T \ext B G = (\ext A \ext B \ext A^T)^{-1}, \ \
  G^T \ext B H = 0, \ \
G^T \ext B^T \ext A^T =  G^T \ext B \ext A^T = \ext A \ext B G = I.
\end{equation}
In other words, $g_l$, the $l$th column of $G$ is perpendicular to $Bh_m$, with $h_m$  the $m$th column of $H$. (Keep in mind that $\ext B$ is diagonal.)
Let $\invw_l:= \{w\in \invw: \pi_l=0\}$ denote the $l$th facet of $\invw$, we  see that $g_l$ is  perpendicular to $\invw_l$.
The ${\sum_{r\in\route}F_r}$-dimensional  matrix $(G,H)$ is invertible (cf.\ \cite{YeYao2012}), hence we can  decompose the ${\sum_{r\in\route}F_r}$-dimensional vector $w$ as
\begin{equation}
  \label{eq:w-decompose-GH}
  w = \ext B G y + \ext B H z.
\end{equation}
It follows from \eqref{eq:H-matrix}, \eqref{eq:w-decompose-AH} and \eqref{eq:G-matrix} that
\begin{equation}
  \label{eq:z-pi}
  H^Tw=z \ \textrm{ and }\ G^Tw=\pi.
\end{equation}

\paragraph{Dynamic complementarity problem.}
Consider the following dynamic complementarity problem (DCP), also known as Skorokhod problem. \begin{align}
  \label{eq:dcp-main}
  & \ext w(t) = \ext w(0) + \ext x(t) + \ext BG y(t) + \ext B H z(t) \geq 0,\\
  & G^T \ext w(t) \geq 0,\label{eq:dcp-posi}\\
  & y_l(t) \textrm{ is nondecreasing in } t \textrm{ with }y(0) = 0,\label{eq:dcp-mono}\\
  & \int_0^\infty \ext w(t)^T G dy(t) = 0,\label{eq:dcp-compli}\\
  & H^T \ext w(t) = 0. \label{eq:dcp-H}
\end{align}
If we multiply \eqref{eq:dcp-main} by $H^T$ from the left, we have $z(t) = -H^Tx(t)$ due to \eqref{eq:H-matrix} and \eqref{eq:G-matrix}. Also note that \eqref{eq:w-decompose-AH} and \eqref{eq:z-pi} imply that
\begin{equation*}
  \ext B \ext A^T G^T + \ext B H H^T = I.
\end{equation*}
Therefore, we can eliminate $z(t)$ in \eqref{eq:dcp-main} to obtain
\begin{equation}
  \label{eq:dcp-main-alt}
  \ext w(t) = \ext w(0) + \ext B \ext A^T G^T \ext x(t) + \ext B G y(t) \geq 0.
\end{equation}
It is pointed out in \cite{YeYao2012} that the DCP problem characterized by \eqref{eq:dcp-main-alt} and \eqref{eq:dcp-posi}--\eqref{eq:dcp-compli} can be transformed to a standard Skorohod problem (e.g., \cite{Williams1998}) if we consider $\ext w_G(t) = G^T \ext w(t)$.
Let $\Phi:\D\to\D^3$ denote the solution to the DCP \eqref{eq:dcp-main}--\eqref{eq:dcp-H}, i.e.,
\begin{equation*}
  (\ext w, y, z) = \Phi(\ext x).
\end{equation*}
The results in this section are required to derive the diffusion limit in Section \ref{sec:diff-appr}.

\paragraph{Reflection on the boundary.}
To connect this DCP with our workload process, observe that, applying the decomposition \eqref{eq:w-decompose-GH}, the dynamics for the stochastic workload process \eqref{eq:dynamics-k-workload} can be written as
\begin{equation}
  \label{eq:dynamics-k-workload-decomp}
  \ext W^k(t) = \ext W^k(0) + \ext X^k(t) + \ext B G Y^K(t) + \ext B H Z^k(t),
\end{equation}
where $Y^k(t)$ is defined in \eqref{eq:Y-k} and
\begin{equation*}
  Z^k(t) = H^T \int_0^t [\ext\rho - \ext \Lambda(\ext N^k(s))] ds.
\end{equation*}
We see that \eqref{eq:dcp-main} and \eqref{eq:dcp-mono} are valid, while in general \eqref{eq:dcp-posi}, \eqref{eq:dcp-compli} and \eqref{eq:dcp-H} are not.
A main technical challenge of the paper is to show that they are approximately valid for large $k$ under a heavy traffic assumption.

The condition \eqref{eq:dcp-H} says  $\ext w$ lives in $\invw$.
This is not the case in the pre-limit, but  if $w$ is close to $\invw$ and there is backlog at link $l$, then that link is working
at full capacity, which is approximately \eqref{eq:dcp-compli}. To make this formal, define the distance from any state $w$ to $\invw$ as
\begin{equation*}
  \dist(w) = \sum_{l\in\link} (-g_l^Tw)^+ + \sum_{m=1}^{{\sum_{r\in\route}F_r}-L} |h_m^T w|.
\end{equation*}
The intuition behind this definition is that, following from \eqref{eq:z-pi}, $w$ is an invariant point if and only if $z=H^Tw=0$ and $\pi\equiv G^Tw\geq 0$.
A key lemma is \cite[Lemma 2]{YeYao2012}.

\begin{lemma}[\cite{YeYao2012}]
  \label{lem:reflection-boundary}
  Let $M>0$ and $\epsilon>0$ be given. There exists a constant $\sigma=\sigma(M, \epsilon)>0$ (sufficiently small) such that
  the following implication holds for any $l\in\link$:
  \begin{equation*}
    g_l^Tw>\epsilon \Rightarrow \ext A_l \ext\Lambda (\ext n) = c_l
  \end{equation*}
  if both $|w|\leq M$ and $\dist(w)\le \sigma$.
\end{lemma}

To make this lemma relevant, we need to guarantee that we come close to $\invw$ in the first place. This motivates the next section, where we introduce and analyze an auxiliary fluid model.

\section{A fluid model and its convergence to equilibrium}
\label{sec:fluid-model}

The goal of this self-contained section is to introduce and analyze a fluid model. We consider a more general setting: rather than analyzing the model at the phase level, we assume there is a general routing matrix $P$ between different routes. Completed jobs from route $r$ have probability $P_{r,r'}$ to be routed to route $r'$. It is clear that this setting is more general than the phase-type model introduced in Section~\ref{sec:network-model}, where routing is only restricted within phases of each route.  This also allows us to simplify the notation in this section.

The overview of the present section is as follows.
\begin{enumerate}
\item We introduce a fluid model for a model with general routing, which is related to the fluid model in \cite{Massoulie2007} - in fact, we add another requirement to the definition of \cite{Massoulie2007}, so that a function which is a fluid model in our sense, also satisfies the requirements in \cite{Massoulie2007}. We introduce an entropy-like function which was shown in \cite{Massoulie2007} to be a Lyapunov function in the sub-critically loaded case.

\item We show that the entropy-like function remains a Lyapunov function under critical loading. This requires a careful analysis; as also stipulated in \cite{Bramson1996b}, who considered subcritical and critical fluid models of head of the line PS systems. As in \cite{Massoulie2007}, we use classical rearrangement inequalities, but we do so in an entirely different way: we show that the derivative of the Lyapunov function can be rewritten as the expected value of a path functional of a terminating Markov chain, for which we obtain pathwise bounds (see proof of Lemma~\ref{lem:ineq-rearrange}). Our arguments would provide a substantial simplification of the subcritical case, as treated in \cite{Massoulie2007}.

\item Using the bound of the derivative of the Lyapunov function, we then proceed to prove uniform convergence of fluid model solutions towards
the invariant manifold leading to Theorem~\ref{thm:fluid-conv-inv}. On a high level, our approach is similar to that of \cite{Bramson1996b}:

\begin{enumerate}
\item Find a function $L$ which is a Lyapunov function, i.e., show that $f(t)=L(n(t))$ has negative derivative bounded by  $-g(n(t))$, with $g$ a nonnegative function.

\item Show that $f(t) \leq |n(t)|K g(n(t))$ for some constant $K$ independent of $n(t)$.

\item  The two inequalities combined give $f'(t) \leq - f(t)/(K|n(t)|)$. By bounding $|n(t)|$ in terms of $|n(0)|$ we get uniform rates of convergence of $f(t)$ to 0, leading to uniform convergence of $n(t)$ for all fluid models starting in a compact set. 

\end{enumerate}
On a more detailed level, our arguments are different.
Apart from simplifying and extending ideas from \cite{Massoulie2007}, we develop and use several additional properties of proportional fairness in the process.
\end{enumerate}
In this section, we use lower case for fluid model quantities, such as $n(t)$, instead of $\bar{N}(t)$.

\subsection{A fluid model}

The definition of the route-level quantities $\lambda, \mu, \rho, P$ are still in force, as is the assumption $A\rho=c$. The routing matrix $P$ is no longer block-diagonal. The two assumptions we invoke are
\begin{align}
  \label{eq:assum_P_inv}
  I-P^T \textrm{ is invertible},\\
  \label{eq:assum_lambda_pos}
  \lambda_r>0 \textrm{ for all } r\in\route.
\end{align}
The latter assumption is required for the analysis in this section to work.
Recall that $\Lambda_r(n(t))$ solves the problem \eqref{eq:opt-route}. We are now in a position to present our definition of a fluid model.

\begin{definition}[Fluid Model]
  \label{def:fluid}
  A fluid model is a function $\{n_{r}(t),t\geq 0\}_{r\in\route}$ is an absolute continuous function such that, for almost every $t$,
  \begin{equation}
    \label{eq:fluid-def}
    \dot{n}_{r}(t) = \lambda_{r} - \mu_r \Phi_{r}(n(t))  + \sum_{s\in\route} P_{s,r} \mu_{s} \Phi_{s}(n(t)),
  \end{equation}
  where
  \begin{equation}
    \label{eq:R-Lambda}
    \Phi_r(n(t)) \left\{
      \begin{array}{ll}
        = \Lambda_r(n(t)), & \textrm{ if } n_r(t)>0,\\
        \in [0, \limsup_{y\rightarrow n(t)}  \Lambda_r(y)], & \textrm{ if }n_r(t)=0,
      \end{array}
      \right.
  \end{equation}
  and
  \begin{equation}
    \label{eq:fluid-def-R-<=}
    \sum_{r\in\route}A_{l,r}\Phi_r(n(t))\le c_l\quad \textrm{for all \ } l\in\link.
  \end{equation}
   The auxiliary functions $w(\cdot)$ and $y(\cdot)$ are defined by
  \begin{align}
    \label{eq:fluid-def-W}
    w(t) &= \diag{m}(I-P^T)^{-1}n(t),\\
    \label{eq:fluid-def-Y}
    y(t) &= A\int_0^t[\rho-\Phi(n(s))] ds.
  \end{align}
\end{definition}
Note that the processes $w(\cdot)$ and $y(\cdot)$ are simply derived from $n(\cdot)$.
We call a function $n(\cdot)$ meeting the requirements of Definition \ref{def:fluid} a fluid model solution.
This definition is essentially the same as the one in \cite{Massoulie2007}, though we also require \eqref{eq:fluid-def-R-<=}. This makes the analysis in the present section more convenient, without increasing the burden much when we need to connect with the original stochastic model. As our fluid model solutions also are fluid model solutions in the sense of \cite{Massoulie2007}, we can exploit properties developed in that work.
We call $t$, where \eqref{eq:fluid-def}--\eqref{eq:fluid-def-R-<=} are satisfied, a regular point. If $t$ is regular, we will often say that the associated state vector $n(t)$ is regular. We now provide a more explicit representation for $\Phi_r(n(t))$ for any regular $t$. Introduce
\begin{equation}
  \label{eq:route-+-0}
  \route_0(t) = \{r\in\route: n_r(t)= 0\}
  \ \textrm{ and } \
  \route_+(t) = \{r\in\route: n_r(t)> 0\}.
\end{equation}
It is clear that for any regular $t$, $n_r(t)= \dot{n}_r(t)=0$ for $r\in \route_0(t)$. This implies that
\begin{equation}
  \label{eq:derivative-regular-0}
  \lambda_r - \mu_r \Phi_r(n(t)) + \sum_{s\in \route_0} \mu_{s} \Phi_{s}(n(t)) P_{s,r}
  + \sum_{s\in \route_+} \mu_{s} \Lambda_{s}(n(t)) P_{s,r}
  = 0,\quad r\in\route_0.
\end{equation}
This gives an affine relationship between $(\Phi_r)_{r\in \route_0}$ and $(\Lambda_r)_{r\in \route_+}$. Such an affine relationship depends on the set $\route_+$, which can take only finitely many different values. Thus, we can derive the scalability of $\Phi$ from that of $\Lambda$, i.e., for any scalar $y>0$,
\begin{equation}
  \label{eq:scalability}
  \Phi(n(t)) = \Phi(yn(t)).
\end{equation}
The main goal of this section is to give a proof of the following result: 
\begin{theorem}
  \label{thm:fluid-conv-inv}
  Assume \eqref{eq:assum_P_inv} and \eqref{eq:assum_lambda_pos}. Let $n(\cdot)$ be a fluid model solution. If $|n(0)|<M$ for some constant $M>0$, then for all $\epsilon>0$ there exist a time $T_{M,\epsilon}$ (not depending on $n(\cdot)$) and a state  $n(\infty)$, such that
  \begin{equation*}
    |n(t)-n(\infty)|<\epsilon \quad\textrm{ for all }\ t>T_{M,\epsilon}.
  \end{equation*}
\end{theorem}
This theorem will be a key tool in the derivation of the diffusion limit in the next section. The remainder of the current section is devoted to its proof.

\subsection{A Lyapunov function}
Introduce
\begin{equation}
  \label{eq:lyapunov}
  L(n(t)) = \sum_{r\in\route} n_r(t) \log \left(\frac{\Phi_r(n(t))}{\rho_r}\right).
\end{equation}
Note that $0\log 0$ is always meant to be $0$.
For convenience, denote $f(t) = L(n(t))$.
We have Lemma~5 of \cite{Massoulie2007}, which we copy almost verbatim.

\begin{lemma}[Basic characterizations from \cite{Massoulie2007}]
  \label{lem:basic-Massoulie}
  Let $n(t)$ be a fluid model solution, and let $\route_0(t)$ and $\route_+(t)$ be as defined in \eqref{eq:route-+-0}.
  \begin{enumerate}
  \item [(i)]
    There exists a constant $M$, such that, for all $t\ge 0$:
    \begin{equation*}
      \limsup_{h\downarrow 0} \frac{f(t+h)-f(t)}{h} \leq  M.
    \end{equation*}
    Let
    \begin{equation}
      \label{eq:lyp-derivative}
      \dot{f}(t) := \sum_{r\in  \route_+(t)} \dot{n}_r(t) \log \left(\frac{\Lambda_r(n(t))}{\rho_r}\right),
    \end{equation}
    then for almost every $t\ge 0$,
    \begin{equation*}
      \limsup_{h\downarrow 0} \frac{f(t+h)-f(t)}{h}\le \dot{f}(t).
    \end{equation*}
  \item [(ii)]
    There exist modified arrival rates $(\tilde \lambda_r)_{r\in \route_+(t)}$ and modified routing probabilities $(\tilde P_{r,s})_{r,s \in \route_+(t)}$, that depend only on the set $\route_0(t)$, such that the matrix $(\tilde P_{r,s})_{r,s \in  \route_+(t)}$ is sub-stochastic with spectral radius strictly less than 1. The identity
    \begin{equation*}
      (\lambda_r)_{r \in  \route_+(t)} = (I-\tilde P^T)^{-1} \tilde \lambda
    \end{equation*}
    holds, and in addition, for almost every $t>0$,
    \begin{equation}
      \label{eq:ODE-modified-rate}
      \dot n_r(t) = \left\{
      \begin{array}{ll}
        \tilde \lambda_r  + \sum_{s\in  \route_+(t)} \mu_s \tilde P_{r,s} \Lambda_s(n(t)) - \mu_r \Lambda_r(n(t)), & r\in \route_+(t),\\
        0, & r\in \route_0(t).
      \end{array}
      \right.
    \end{equation}
    \item [(iii)]
      Let $u_r(t) = \log \left(\frac{\Lambda_r(n(t))}{\rho_r}\right)$ for all $r\in\route_+(t)$.
      \begin{equation}
        \label{eq:lyp-derivative-alt}
        \dot{f}(t) = - \sum_r \lambda_r \sum_{k=0}^{\infty} \sum_{s\in  \route_+(t)} \tilde P_{r,s}^{(k)} (e^{u_s(t)}-1)
        \left[u_s(t) - \sum_{s'\in\route_+(t)} \tilde P_{s,s'} u_{s'}(t) \right].
      \end{equation}
  \end{enumerate}
\end{lemma}
\begin{proof}
  Properties (i) and (ii) follow from Lemma~5 of \cite{Massoulie2007} and property (iii) follows from the arguments on page 821 of \cite{Massoulie2007}.
\end{proof}

In \cite{Massoulie2007}, an elaborated argument is followed to show that $\dot{f}(t)<0$ in the sub-critically loaded case. In this paper, we study the critical loaded case (i.e., $A\rho =c$).
The analyses in these two cases are quite different (cf.\ the difference of complexity between convergence of subcritical and critical fluid models as exhibited in \cite{Bramson1996b}).
From this moment on, our analysis and the analysis in \cite{Massoulie2007} follow separate ways.

\subsection{Bounding the derivative of the Lyapunov function}

\begin{proposition}
  \label{prop:lyp-derivative-bound}
  For any regular $t\ge 0$,
  \begin{equation*}
    \dot{f}(t) \leq - \sum_{r\in  \route_+(t)}  \lambda_r
    \left(
      \frac{\Lambda_r(n(t))}{\rho_r}-1\right) \log \left(\frac{\Lambda_r(n(t))}{\rho_r}
    \right).
  \end{equation*}
  Assuming $\lambda_r$ is strictly positive for all routes $r$, there exists an $\epsilon>0$ such that
  \begin{equation*}
    \dot{f}(t) \leq -
    \epsilon \sum_{r \in \route_+ (t)} \left( \Lambda_r(n(t))-\rho_r\right) \log \left(\frac{\Phi_r(n(t))}{\rho_r}\right).
  \end{equation*}
\end{proposition}

The proof follows directly from Lemma~\ref{lem:basic-Massoulie} and the following lemma based on a rearrangement inequality, which may be of independent interest.

\begin{lemma}
  \label{lem:ineq-rearrange}
  Let $\{u_r\}_{r\in\route_+}$ be arbitrary real numbers where $\route_+$ is any subset of positive integers. Set
  \begin{equation*}
    h_r =  \sum_{k=0}^{\infty} \sum_{s\in \route_+} \tilde P_{r,s}^{(k)} (e^{u_s}-1)\left[u_s - \sum_{s'\in \route_+} \tilde P_{s,s'} u_{s'} \right].
  \end{equation*}
  Then
  \begin{equation*}
    h_r \geq u_r (e^{u_r}-1).
  \end{equation*}
\end{lemma}

\begin{proof}
Let $X_k$ be a Markov chain on $\route_+ \cup \{0\}$ starting from $X_0=r$ evolving according to the transition matrix $\tilde P$ with 0 as absorbing state. Set $h_0=0$ and $u_0=0$.
Note that
\begin{equation*}
  \pr(X_k=s, X_{k+1}=s' \mid X_0=r) = \tilde P_{r,s}^{(k)} \tilde P_{s,s'}.
\end{equation*}
Let $\E_r[\cdot]$ denote the conditional expectation given that $X_0=r$.
Set $v_r = e^{u_r}-1$ for all $r\in\route_+\cup\{0\}$, then
\begin{align*}
  h_r &= \sum_{k=0}^\infty\sum_{s\in \route_+} \tilde P_{r,s}^{(k)} v_s (u_s-\sum_{s'\in \route_+} \tilde P_{s,s'}u_s')
  = \sum_{k=0}^\infty\sum_{s\in \route_+} \sum_{s'\in \route_+} \tilde P_{r,s}^{(k)} \tilde P_{s,s'} v_s (u_s-u_{s'})\\
  &= \sum_{k=0}^\infty \E_r\left[v_{X_k}(u_{X_k}-u_{X_{k+1}})\right].
\end{align*}
%
%
Let $k_0=\inf \{n: X_k=0\}$, then
\begin{equation*}
  h_r = \E_r\left[\sum_{k=0}^{k_0-1}  v_{X_k}(u_{X_k}-u_{X_{k+1}}) \right]
      = \E_r \left[\sum_{k=0}^{k_0-1}v_{X_k}(u_{X_k}\id{n>0}-u_{X_{k+1}}) \right]
      + v_r u_r.
\end{equation*}
We claim that, a.s.,
\begin{equation*}
  \sum_{k=0}^{k_0-1}v_{X_k}u_{X_k}\id{n>0} \geq \sum_{k=0}^{k_0-1} v_{X_k} u_{X_{k+1}}.
\end{equation*}
This follows from a classical rearrangement inequality in \cite{HLP1988} stating that if $(a_k)$ and $(b_k)$ are two non-decreasing finite sequences, and $(b_k^{[p]})$ is a permutation of $(b_k)$, then $\sum_{k} a_kb_k \geq \sum_{k} a_k b_k^{[p]}$. We can apply this inequality since $0=u_{X_0}\id{0>0} = u_{X_{k_0}}$. Thus, $h_r \geq v_r u_r$ and the lemma is proved.
\end{proof}

\subsection{Bounding the Lyapunov function in terms of its derivative}

Having established an upper bound for $\dot{f}(t)$, our next task is to connect this bound to $f(t)$, which is establish in the next proposition.

\begin{proposition}
  \label{prop:lyp-bound}
  Let $\epsilon$ be given in Proposition~\ref{prop:lyp-derivative-bound}.
  There exists $\zeta^*<\infty$ such that for almost every $t\ge 0$,
  \begin{equation*}
    \dot{f}(t) \leq -\frac{\epsilon}{\zeta^*} \frac{1}{|n(t)|} f(t).
  \end{equation*}
\end{proposition}


Define $p(t) = \frac{n(t)}{|n(t)|}$ with the convention that $0/0=0$.
By the scalability of $\Phi$ in \eqref{eq:scalability}, $\Phi_r(n(t))=\Phi_r(p(t))$.
%
According to \eqref{eq:R-Lambda} and \eqref{eq:fluid-def-R-<=}, if $t$ is a regular point, then $(\Lambda_r(n(t)))_{r\in \route_+(t)}$ also solves the optimization problem
\begin{equation}
\label{eq:opt-modified}
\max_\gamma \sum_{r\in \route_+(t)} n_r(t) \log(\gamma)
\end{equation}
subject to
\begin{equation}
\label{eq:opt-constr-modified}
\sum_{r\in \route_+} A_{l,r} \gamma_r \leq c_l - \sum_{r\in \route_0} A_{l,r} \Phi_r(p(t)) =: c_l(p(t)).
\end{equation}
Let $(\eta(p(t)))_{l\in\link}$ be the Lagrange multipliers satisfying the Karush-Kuhn-Tucker (KKT) conditions (c.f.\ Section~5.5.3 in \cite{BoydVandenberghe2004})  associated with the optimization problem \eqref{eq:opt-route}, and define
\begin{equation*}
   \zeta_r(p) = \sum_l A_{lr} \eta_l(p).
\end{equation*}
In the following lemma, we assume that all objects are at a regular time $t\ge 0$. Thus, we omit the parameter $t$ for notational simplicity.
\begin{lemma}
  \label{lem:ration-bound}
 For any $r$,
  \begin{equation*}
    \sup_{p: |p|=1, A\Phi(p)\leq c} \zeta_r(p) < \infty.
  \end{equation*}
\end{lemma}

\begin{proof}
It follows from \eqref{eq:derivative-regular-0} and condition \eqref{eq:assum_lambda_pos} that $\Phi_r(p )\ge\frac{\lambda_r}{\mu_r} > 0$, for all $r\in\route_0 $.
for all regular $t\ge 0$, define
\begin{equation*}
  \link_0  = \{l\in\link: A_{r,l}>0 \textrm{ for some }r\in\route_0 \}.
\end{equation*}
Then
\begin{equation}
  \label{eq:tech-c-strict}
  c_l(p )<c_l
\end{equation}
for all $l\in\link_0 $.
We can see $\eta_l(p )=0$ for all $l\in\link_0 $ since the $l$th constraint in \eqref{eq:feasible-set-route} is not binding due to \eqref{eq:tech-c-strict} in this case. This implies $\zeta_r(p)=0$ for any regular $p$ with $p_r=0$.
%
To handle cases where $p_r>0$, we use duality.
Let
\begin{equation*}
  \link_+=\{l\in\link: A_{r,l}>0\textrm{ for some }r\in\route_+\}.
\end{equation*}
Note that both $\route_+$ and $\link_+$ are nonempty. Moreover,
\begin{equation}
  \label{eq:positive-residual-cap}
  c_l(p)>0 \textrm{ for all } l\in\link_+.
\end{equation}
The Lagrangian of the optimization problem \eqref{eq:opt-modified} with \eqref{eq:opt-constr-modified} can be written as
\begin{equation}
  \label{eq:lagrangian}
  \max_{\gamma_r,r\in \route_{+}} \sum_{r\in \route_+} p_r \log \gamma_r
  - \sum_{l\in \link_+} \eta_l \left(\sum_{r\in \route_+} A_{l,r} \gamma_r - c_l(p)\right).
\end{equation}
By the optimality condition
$\frac{p_r}{\gamma_r} = \sum_{l\in \link_+}  A_{l,r}\eta_l, \quad \textrm{ for all } r\in \route_+$
we obtain
\begin{align*}
  \sum_{l \in \link_+} \eta_l \sum_{r \in \route_+} A_{l,r} \frac{ p_r}{\sum_{{l' \in \link_+}} \eta_{l'} A_{l',r}}
  = \sum_{r \in \route_+} p_r \frac{\sum_{l \in \link_+} \eta_l  A_{l,r}}{\sum_{{l' \in \link_+}} \eta_{l'} A_{l',r}}
  = 1.
\end{align*}
So \eqref{eq:lagrangian} can be simplified as
\begin{align*}
\sum_{r \in \route_+} p_r \log p_r - \sum_{r \in \route_+} p_r \log\Big(\sum_{l \in \link_+} \eta_l A_{l,r}\Big) + \sum_{l \in \link_+} \eta_l c_l(p)
-1.
\end{align*}
By duality, $\eta_l(p)$ solves the optimization problem
\begin{equation*}
  \inf_{\eta \geq 0} \left( \sum_{l \in \link_+} \eta_l c_l(p) - \sum_{r \in \route_+} p_r \log \Big(\sum_{l \in \link_+} \eta_l A_{l,r}\Big)\right)
\end{equation*}
which is equivalent to, using $\sum_{r \in \route_+} p_r=1$,
\begin{equation*}
  \sup_{\eta \geq 0} \sum_{r \in \route_+} p_r \left[ \log \Big(\sum_{l \in \link_+} \eta_l A_{l,r}\Big) -  \sum_{l \in \link_+} \eta_l c_l(p) \right].
\end{equation*}
It follows from \eqref{eq:positive-residual-cap} that $\left[ \log (\sum_{l\in \link_+} \eta_l A_{l,r}) -  \sum_{l\in \link_+} \eta_l c_l(p) \right]$ is negative when $\eta$ is outside a compact set.
This implies that $\eta(p)$ is necessarily uniformly bounded in $p$ for any fixed $\link_+$.
Since there are only finite choices ($2^{L}$) for $\link_+$, we must have $\sup_{p: |p|=1} |\eta(p)| <\infty$.
\end{proof}

\begin{proof}[Proof of Proposition~\ref{prop:lyp-bound}]
Let $t$ be a regular point.
By Lemma~\ref{lem:ration-bound}, let $\zeta^*<\infty$ be an upper bound of $\zeta_r(p(t))$ for all $p(t)$ such that $|p(t)|=1$ and $A\Phi(p(t))\leq c$.
Using \eqref{eq:scalability} and Proposition~\ref{prop:lyp-derivative-bound}, we have
\begin{align}
  \dot{f}(t)
  &\leq -\epsilon \sum_{r \in \route_+} (\Phi_r(p(t)) - \rho_r) \log \left(\frac{\Phi_r(p(t))}{\rho_r}\right)\nonumber\\
  &\leq -\frac{\epsilon}{\zeta^*} \sum_{r \in \route_+} \zeta_r(p(t)) (\Phi_r(p(t)) - \rho_r) \log \left(\frac{\Phi_r(p(t))}{\rho_r}\right)
  \label{eq:tech-lyp-der-bd}
\end{align}
In the proof, $\epsilon$ may change from step to step but remains strictly positive.
By the KKT conditions, $p_r(t) = \zeta_r(p(t)) \Phi_r(p(t))$ for all $r\in\route_+(t)$. Define $q_r(t) = \zeta_r(p(t))\rho_r$ for all $r\in\route$. Observe that $q_r(t)=0$ for all $r\in\route_0(t)$ due to Lemma~\ref{lem:ration-bound}.
Then \eqref{eq:tech-lyp-der-bd} becomes
\begin{equation}
  \label{eq:tech-lyp-der-bd-2}
  -\frac{\epsilon}{\zeta^*} \sum_{r \in \route_+(t)} \left(p_r(t) - q_r(t)\right)
  \log \left(\frac{\Phi_r(p(t))}{\rho_r}\right).
\end{equation}
Consider now the allocation $\Lambda(q)$, which is the solution to the program $\max_{\gamma} \sum_{r\in\route} q_r \log \gamma_r$ subject to $A\gamma \leq c$ and $\gamma_r=0$ if $q_r=0$.
The KKT conditions then read $q_r/\Lambda_r(q) = \sum_{l\in\link} A_{l,r} \eta_l(q)$, $\eta(q) (A\Lambda - c)=0$ for some $\eta(q)\ge 0$.
Since the network is critically loaded, i.e., $A\rho=c$, we may take the Lagrange multipliers $\eta(q)=\eta(p)$, and $\Lambda_r(q)=\rho_r$ if $q_r>0$.
From this, it follows that
\begin{equation*}
  \sum_{r\in \route_+(t)} q_r(t) \log \Phi_r(p(t)) \leq \sum_{r\in \route_+} q_r(t) \log \rho_r.
\end{equation*}
This together with \eqref{eq:tech-lyp-der-bd-2} implies
\begin{align*}
  \dot{f}(t)
  \leq -\frac{\epsilon}{\zeta^*} \sum_{r\in\route} p_r(t) \log \left(\frac{\Phi_r(p(t))}{\rho_r}\right)
  = -\frac{\epsilon}{\zeta^*}  \frac{1}{|n(t)|} f(t).
\end{align*}
\end{proof}

\subsection{Compactness and convergence to invariant manifold}

We first derive some additional properties of $f$, with the goal of connecting the end of our proof with \cite{Bramson1996b}.

\begin{proposition}
  \label{prop:further-bound}
  \begin{align*}
    f(0) &=  \sum_{r\in\route} n_r(0) \log \left(\frac{\Phi_r(n(0))}{\rho_r}\right)
    \leq |n(0)| \log \left( \frac{\max_lc_l}{\min_{r} \rho_{r}}\right),\\
    \dot{f}(t) &\leq -\epsilon \sum_{r\in\route} \left(\frac{\Phi_r(n(t))}{\rho_r}-1\right)^2,
  \end{align*}
  for some $\epsilon>0$ and almost every $t$.
\end{proposition}

\begin{proof}
The first inequality is trivial. The second inequality is derived in two steps. Let $t$ be a regular point. We first note that
\begin{equation}
  \label{eq:tech-partial-+}
  \dot{f}(t) \leq -\epsilon \sum_{r\in \route_+} \left(\frac{\Phi_r(n(t))}{\rho_r}-1\right)^2,
\end{equation}
following from Proposition~\ref{prop:lyp-derivative-bound} and the inequality $(a-b)\log (a/b) \geq (a-b)^2/\max \{a,b\}$.
Again, the exact value of $\epsilon$ may change from step to step, but it will always be strictly positive.
The challenge is to extend this to the entire index set $r$, a task the rest of this proof is devoted to.

Set $d_r(t) = \mu_r\Phi_r(x(t))$. We see that
\begin{equation*}
  \dot{f}(t)\leq -\epsilon \sum_{r\in \route_+} \left(d_r(t)-\gamma_r\right)^2,
\end{equation*}
Note that
\begin{equation*}
  d_r(t) = \lambda_r + \sum_{r'\in\route} P_{r',r} d_{r'}(t)\quad \textrm{ for all }\ r\in\route_0(t).
\end{equation*}
On the other hand, we have
\begin{equation*}
  \gamma_r = \lambda_r + \sum_{r'\in\route} P_{r',r} \gamma_{r'}\quad \textrm{ for all }\ r\in\route.
\end{equation*}
So for all $r\in\route_0(t)$,
\begin{equation*}
  d_r(t)-\gamma_r = \sum_{r'\in\route} P_{r',r} (d_{r'}(t)-\gamma_{r'}).
\end{equation*}
We use this expression to say something about the vector $(d(t)-\gamma)|_{\route_0(t)}$, which is is formed by the coordinates of the vector $d(t)-\gamma$ corresponding to those coordinates $r\in \route_0(t)$.
Let $P^{0,0}$ be the matrix built up from all routing probabilities from $\route_0(t)$ to $\route_0(t)$ and let $P^{+,0}$ be the matrix consisting of routing probabilities from states $\route_+(t)$ to $\route_0(t)$.
Then
\begin{equation*}
  (d(t)-\gamma)|_{\route_0} = P^{0,0} (d(t)-\gamma)|_{\route_0} + P^{+,0} (d(t)-\gamma)|_{\route_+}.
\end{equation*}
Since $I-P$ is invertible, so is $I-P^{0,0}$ (where $I$ is of appropriate dimension) and we see that
\begin{equation*}
  (d(t)-\gamma)|_{\route_0} = (I-P^{0,0})^{-1} P^{+,0}(d(t)-\gamma)|_{\route_+} =: P^\dagger (d(t)-\gamma)|_{\route_+}.
\end{equation*}
The matrix $P^\dagger$ consists of nonnegative elements.
We conclude that for $r\in \route_0$,
\begin{equation}
  d_r(t) - \gamma_r = \sum_{r'\in \route_+} P^\dagger_{r'r} (d_{r'}(t)-\gamma_{r'}) \label{Phi-boundary}.
\end{equation}
The Cauchy-Schwarz inequality yields
\begin{align*}
   (d_r(t) - \gamma_r)^2
   \le  \sum_{r'\in \route_+} {P^\dagger_{r'r}}^2 (d_{r'}(t)-\gamma_{r'})^2
   \leq \|P^\dagger\|_\infty^2 \sum_{r'\in \route_+}  (d_{r'}(t)-\gamma_{r'})^2,
\end{align*}
where $\|P^\dagger\|_\infty$ denotes the largest element in the matrix $P^\dagger$.
Summing up over $r\in \route_0(t)$ yields
\begin{equation*}
  \sum_{r'\in \route_0(t)}  (d_{r'}(t)-\gamma_{r'})^2
  \leq
  \|P^\dagger\|_\infty^2 R \sum_{r'\in \route_+(t)}  (d_{r'}(t)-\gamma_{r'})^2.
\end{equation*}
Combining the above inequality and \eqref{eq:tech-partial-+} leads to the second inequality of this proposition.
\end{proof}


\begin{proof}[Proof of Theorem~\ref{thm:fluid-conv-inv}]
Bramson's proof of his Proposition~6.1 also applies to our setting if we set $d_r(t) = \left(\frac{\Phi_r(n(t))}{\rho_r}-1\right)$,
and the same holds for his {Proposition~6.2}, using Proposition 5.3 at various points in his line of argument. We omit the details.
This guarantees the existence of a constant $B$ such that for all $t\geq 0$,
\begin{equation*}
  |n(t)|\leq B |n(0)|.
\end{equation*}
Following from the above and Proposition~\ref{prop:lyp-bound}, there exists $\epsilon>0$ such that
\begin{equation*}
  f(t) \leq f(0) \exp\{-\epsilon t/|n(0)|\}.
\end{equation*}
From (6.26)--(6.28) of \cite{Bramson1996b} we then obtain that
\begin{equation*}
  |n(t) - n(t')| \leq B |n(0)| \exp\{ - \epsilon t/|n(0)|\},
\end{equation*}
for appropriate constants $\epsilon, B$, for all $t>0$ and $t'>t$. Consequently, $n(t), t\geq 0$ is a Cauchy sequence, and converges to some $n(\infty)$.
The last equation implies
\begin{equation*}
  |n(t) - n(\infty)| \leq B |n(0)| \exp\{ - \epsilon t/|n(0)|\}.
\end{equation*}
i.e.\ convergence is exponentially fast, u.o.c.\ in $|n(0)|$.
Since $f(x)$ is lower semi-continuous (cf.\ Theorem 1 in \cite{Massoulie2007}), we see that
\begin{equation*}
  0\leq L(n(\infty)) \leq \liminf_{t\rightarrow\infty} L(n(t)) = \lim_{t\rightarrow\infty} f(t)=0.
\end{equation*}
Consequently,
\begin{equation*}
  0 = \sum_r n_r(\infty) \log (\Phi_r(n(\infty))/\rho_r)
  = \sum_{r: n_r(\infty)>0} n_r(\infty) \log (\Lambda_r(n(\infty))/\rho_r).
\end{equation*}
Since $\sum_r n_r(\infty) \log (\Lambda_r(n(\infty))) \geq \sum_r n_r(\infty) \log (\Lambda_r')$
for any feasible $\Lambda'$, since $\Lambda (n(\infty))$ is the unique optimum of the PF utility maximization problem, it follows that $\Phi_r(n(\infty))= \Lambda_r(n(\infty)) = \rho_r$ if $n_r(\infty)>0$.
If $n_r(\infty)=0$, an additional argument is needed to show that $\Phi_r(n(\infty))=\rho_r$.

Observe that $n(\infty)$ is an invariant point, since $n(t)$ and $n(t+s)$ both converge to $n(\infty)$ for every fixed $s$ as $t\rightarrow\infty$, and $(n(t+s))_s$ can be seen as time-shifted fluid model with starting point $n(t)$. Since fluid model solutions are regular almost everywhere, a fluid model solution with starting position $n(\infty)$ is regular everywhere. This enables us to apply equation \eqref{Phi-boundary} with $t=\infty$ to conclude that $\Phi_r(n(\infty))=\rho_r$ when $n_r(\infty)=0$.
Consequently, $n(\infty)$ is on the invariant manifold.
\end{proof}

\section{Diffusion approximations}
\label{sec:diff-appr}

The main objective of this section is to study the network  in  heavy traffic to establish the diffusion approximation, stated in Theorem~\ref{thm:diffusion} below.
The main difficult is that the DCP in Section~\ref{sec:geometry-fixed-point} does not hold for the  stochastic system, however it holds only asymptotically in the heavy traffic regime, in a sense we make precise later on.
To this end, we establish  state space collapse (SSC) in Section~\ref{sec:state-space-collapse}, which shows that the diffusion scaled workload process will be close to the invariant manifold and the DCP is satisfied asymptotically (Proposition~\ref{prop:ssc-ac}(ii)).
Using the framework of \cite{Bramson1998}, we prove SSC using a uniform fluid approximation  shown in \ref{sec:unif-fluid-appr}, and the convergence to the invariant state of the fluid model as we have shown in Section~\ref{sec:fluid-model}.


Our heavy-traffic assumption is, as $k\to\infty$,
\begin{align}
  \label{eq:fluid-HT}
  & \lambda^k \to \lambda,\\
  \label{eq:diffusion-HT}
  & k(\rho-\rho^k)\to\theta,
\end{align}
for some $\lambda$ and $\theta\in\R_+^R$.
By \eqref{eq:rho-phase}, this implies $k(\ext\rho_{r,f}-\ext\rho^k_{r,f})\to\ext\theta_{r,f}$ for some $\ext\theta_{r,f}\ge 0$ as $k\to\infty$.
The diffusion scaling is defined as
\begin{equation*}
  \ds{\ext N}(t) = \frac{1}{k} \ext N^k(k^2t),
  \quad
  \ds{\ext W}(t) = \frac{1}{k} \ext W^k(k^2t),
\end{equation*}
and the diffusion scaling for the process quantities is defined as
\begin{equation*}
  \ds{\ext E}(t) = \frac{1}{k}\breve{\ext E}^k(k^2t),\
  \quad
  \ds{\ext S}(t) = \frac{1}{k}\breve{\ext S}^k(k^2t).
\end{equation*}
The definition of the scaling for the corresponding route-level quantities are defined in exactly the same way. Following the above definition, we have the following diffusion scaling
\begin{align}
  \label{eq:X-center-scaled}
  \ds{\ext X}(t) &= \diag{\ext m}(I-\ext P^T)^{-1}\ds{\ext E}(t)
  + k(\ext\rho^k - \ext\rho) t\nonumber\\
  &\quad + \diag{\ext m} (I-\ext P^T)^{-1}
           \left[
             \ds{\ext S}_+\big(\fss{\ext D}(t)\big)-\ds{\ext S}_-\big(\fss{\ext D}(t)\big)
           \right],\\
  \ds{\ext Y}(t) &= \frac{1}{k}\ext A\int_0^{k^2t} [\ext \rho - \ext \Lambda (\ext N^k(s)] ds,\nonumber\\
  \ds{\ext Z}(t) &= \frac{1}{k}\ext H^T\int_0^{k^2t} [\ext \rho - \ext \Lambda (\ext N^k(s)] ds,\nonumber
\end{align}
where $\fss{\ext D}(t)=\ext D^k(k^2 t)/k^2$.
The diffusion scaled process still satisfies the dynamic equation \eqref{eq:dynamics-k-workload-decomp}.
We will not copy it, but later refer to it as the diffusion scaled version of \eqref{eq:dynamics-k-workload-decomp}.

\begin{theorem}
  \label{thm:diffusion}
  Assume conditions \eqref{eq:cond-phase-route-init}, \eqref{eq:P-phase-inv}, \eqref{eq:multi-bottleneck} and \eqref{eq:fluid-HT}--\eqref{eq:diffusion-HT} and the diffusion scaled initial state converges weakly as $k\to\infty$
  \begin{equation}
    \label{eq:initial-cond}
    \ds{\ext W}(0) \dto \chi_0 \in\invw.
  \end{equation}
  The stochastic processes under the proportional fair allocation policy converge weakly as $k\to\infty$
  \begin{equation*}
    \left(\ds{\ext X}(\cdot),\ds{\ext W}(\cdot), \ds Y(\cdot), \ds Z(\cdot)\right)
    \dto
    \left(\hat{\ext X}(\cdot),\Phi(\hat{\ext X})\right),
  \end{equation*}
  where $\hat{\ext X}(\cdot)$ is a Brownian motion with drift $-\ext\theta$ and covariance matrix
  \begin{equation}
    \label{eq:COV}
    \Sigma_X = \diag{\ext m} (I-\ext P^T)^{-1}\left( \diag {\lambda \ext a} + \Sigma_U\right) (I-\ext P)^{-1} \diag{\ext m},
  \end{equation}
  where
  \begin{equation}
    \label{eq:COV-S}
    \Sigma_U = \diag{ (I + \ext P^T) (\ext \rho\cdot \ext{\mu})} -\ext P^T \diag {\ext \rho\cdot \ext{\mu}} - \diag {\ext \rho \cdot \ext{\mu}} \ext P,
  \end{equation}
  $(\lambda \ext{a})_{r,f}=\lambda_r\ext{a}_{r,f}$ and $(\ext{\rho} \cdot \ext{\mu})_{r,f}=\ext{\rho}_{r,f} \ext{\mu}_{r,f}$.
\end{theorem}
The proof of this theorem is postponed to the end of this section.

\subsection{Uniform fluid approximations}
\label{sec:unif-fluid-appr}

We follow the approach and terminology of \cite{Bramson1998}.
The shifted fluid scaling for ``status'' quantities is defined as
\begin{equation*}
  \sfs{j}{U}(t) = \frac{1}{k} U^k(kj+kt)
\end{equation*}
where $U^k$ could be any of the processes $\ext N^k$, $\ext W^k$, $\ext D^k$ and $Y^k$.
The shifted fluid scaling for ``process'' quantities is defined as
\begin{equation*}
  \sfs{j}{U}(t) = \frac{1}{k} [U^k(kj+kt)-U^k(kj)]
\end{equation*}
where $U^k$ is a symbolic notation for $\breve{\ext E}^k$ and $\breve{\ext S}^k$.
%
To connect the shifted fluid scaling and diffusion scaling, consider the diffusion scaled process on the interval $[0, T]$, which corresponds to the interval $[0, k^2T]$ for the unscaled process. Fix a constant $L>1$, the interval will be covered by the $\fl{kT} + 1$ overlapping intervals
\begin{equation*}
  [kj,kj+kL],\quad j=1,2,\ldots, \fl{kT}.
\end{equation*}
For each $t\in[0,T]$, there exists a $j\in\{0,...,\fl{kT}\}$ and $s\in[0,L]$ (which may
not be unique) such that $k^2t = kj + ks$. Thus,
\begin{equation}
  \label{eq:diffusion-fluid}
  \ds X(t) = \sfs{j}{X}(s).
\end{equation}

To utilize the shifted fluid scaled processes to analyze the diffusion scaled processes, we present a uniform fluid approximation, which is similar to \cite[Lemma~12]{YeYao2012}.

\begin{proposition}
  \label{prop:fluid-approx}
  Assume \eqref{eq:fluid-HT} and the existence of $M>0$ such that the initial state $|\bar{\ext N}^{k,j_k}(0)|<M$ for all $k$, where $j_k$ is an integer in $[0,kT]$.
  For any subsequence of $\{k\}_0^\infty$ there exists subsequence $\mathcal K$ along which $(\bar{\ext N}^{k,j_k},\bar{\ext W}^{k,j_k},\bar{\ext D}^{k,j_k},\bar{Y}^{k,j_k})$ converges with probability 1 \uoc to the fluid limit $(\bar{\ext N},\bar{\ext W},\bar{\ext D},\bar{Y})$ that satisfies the fluid model equations \eqref{eq:fluid-def}--\eqref{eq:fluid-def-Y}.
\end{proposition}

\begin{proof}
   Following \cite[Proposition~4.2]{Bramson1998} and \cite[Appendix~A.2]{Stolyar2004}, using Chebyshev's inequality and the Borel-Cantelli lemma, we have that, as $k\to\infty$,
  \begin{align*}
    &\sup_{s\in[0,kT]}\sup_{t\in[0,L]}|\frac{1}{k}\breve{\ext E}^k(ks+kt)|\to 0,\\
    &\sup_{s\in[0,kT]}\sup_{t\in[0,L]}|\frac{1}{k}\breve{\ext S}^k(ks+kt)|\to 0,
  \end{align*}
  a.s.\ (almost surely) for any fixed $T>0$ and $L>0$. This implies that a.s.\, as $k\to\infty$,
  \begin{align*}
    \max_{j\in kT}\sup_{t\in[0,L]}\big(\sfs{j}{\breve{\ext E}}(t),\sfs{j}{\breve{\ext S}}(t)\big)
    \to (\ext 0,\ext 0).
  \end{align*}
  From this point, we can apply exactly the same approach as in \cite[Appendix~A.1]{Massoulie2007} to obtain the fluid approximation. Applying the shifted fluid scaling to the system dynamics equations \eqref{eq:dynamics-k} and \eqref{eq:accum-allocation} and the scalability of $\ext\Lambda_{r,f}(\cdot)$, we have
  \begin{align*}
    \sfs{j}{\ext N}_{r,f}(t) &= \sfs{j}{\ext N}_{r,f}(0) + \lambda_r \ext{a}_{r,f}t
    + \sum_{f'\in\phase_r}\ext{\mu}_{r,f'} P^r_{f',f} \int_0^t \ext \Lambda_{r,f'}\big(\sfs{j}{\ext N}(s)\big)ds\\
    & \quad - \ext{\mu}_{r,f} \int_0^t \ext \Lambda_{r,f}\big(\sfs{j}{\ext N}(s)\big)ds
    + \fs\epsilon_{r,f}(t),
  \end{align*}
  where, recalling the notations defined in \eqref{eq:center-E}--\eqref{eq:center-S--},
  \begin{align*}
    \sup_{t\in[0,L]}|\fs\epsilon_{r,f}(t)|
    &\le \sup_{s\in[0,kT]}\sup_{t\in[0,L]}\frac{1}{k}|\breve{\ext E}^k_{r,f}(ks+kt)|\\
    &\quad +\sup_{s\in[0,kT]}\sup_{t\in[0,L]}\frac{1}{k}\sum_{f'\in\phase\cup\{0\}}|\breve{\ext S}^k_{r,f,f'}(ks+kt)|
    +\sup_{s\in[0,kT]}\sup_{t\in[0,L]}\frac{1}{k}\sum_{f'\in\phase}|\breve{\ext S}^k_{r,f',f}(ks+kt)|.
  \end{align*}
  This implies $\sup_{t\in[0,L]}|\fs\epsilon_{r,f}(t)| \to 0$ a.s. as $k\to\infty$.
  Since we assume that $|\bar{\ext N}^{k,j}(0)|<M$ for all $j,k$, by a variation of the Arzela-Ascoli theorem (see \cite[Lemma~6.3]{YeOuYuan2005}), for any sub-sequence there exists a further sub-sequence such that, as $k\to\infty$, almost surely,
  \begin{align}
    \label{eq:tech-conv-to-D}
    \int_0^t \ext \Lambda_{r,f}\big(\sfs{j}{\ext N}(s)\big)ds &\to \bar{\ext D}_{r,f'}(t)\quad \textrm{\uoc{} on }[0,L],\\
    \sfs{j}{\ext N}_{r,f}(t)(t) &\to \bar N(t)\quad \textrm{\uoc{} on }[0,L],\nonumber
  \end{align}
  where
  \begin{equation*}
    \bar{\ext N}_{r,f}(t) = \bar{\ext N}_{r,f}(0) + \lambda_r \ext{a}_{r,f}t
    + \sum_{f'\in\phase_r}\ext{\mu}_{r,f'} P^r_{f',f} \bar{\ext D}_{r,f}(t)
    - \ext{\mu}_{r,f}\bar{\ext D}_{r,f}(t).
  \end{equation*}
  To avoid complicating the notation, we still use $k$ to index the sub-sequence. By Rademacher's theorem, $\bar{\ext D}_{r,f}(t)$ is differentiable almost every where on $[0,L]$. For any differentiable point $t$, if $\bar{\ext N}_{r,f}(t)>0$, then $\Lambda_{r,f}(\cdot)$ is continuous at $\bar{\ext N}(t)$ according to \cite[Lemma~6.2(b)]{YeOuYuan2005}. Thus, there exists an $h>0$ such that $\bar{\ext N}_{r,f}(s)>0$ for all $s\in[t,t+h]$ and as $k\to\infty$,
  \begin{equation*}
    \int_t^{t+h} \ext \Lambda_{r,f}\big(\sfs{j}{\ext N}(s)\big)ds
    \to
    \int_t^{t+h} \ext \Lambda_{r,f}\big(\bar{\ext N}(s)\big)ds.
  \end{equation*}
  If $\bar{\ext N}_{r,f}(t)=0$, then by Fatou's lemma,
  \begin{equation*}
    \lim_{k\to\infty}\int_t^{t+h} \ext \Lambda_{r,f}\big(\sfs{j}{\ext N}(s)\big)ds
    \le
    \int_t^{t+h} \limsup_{y\to\bar{\ext N}(s)}\ext \Lambda_{r,f}(y)ds.
  \end{equation*}
  On the other hand, the function $x\to\limsup_{y\to x}\ext \Lambda_{r,f}(y)$ is upper semi-continuous, thus
  \begin{equation*}
    \limsup_{s\to t}\limsup_{y\to\bar{\ext N}(s)}\ext \Lambda_{r,f}(y)
    \le \limsup_{y\to\bar{\ext N}(t)}\ext \Lambda_{r,f}(y).
  \end{equation*}
  This implies that the derivative of $\bar{\ext D}_{r,f}(t)$ at $t$ must lie in the interval $[0,\limsup_{y\to\bar{\ext N}(t)}\ext \Lambda_{r,f}(y)]$. This is why we construct the extension of $\Lambda(\cdot)$ as $\Phi(\cdot)$ to be the derivative of $\bar{\ext D}_{r,f}(t)$ (see \eqref{eq:R-Lambda} in Definition \ref{def:fluid}). It now remains to show that
  \begin{equation}
    \label{eq:tech-A-Phi-verify}
    \ext A\Phi(\bar{\ext N}(t))\le c.
  \end{equation}
  Observing that $\ext A\ext \Lambda(\bar{\ext n})\le c$ for any state $\ext n$ due to the allocation policy \eqref{eq:opt-route} we conclude for the pre-limit process $\bar{\ext N}(\cdot)$ that
  \begin{equation*}
    \int_t^{t+h} \ext A \ext \Lambda_{r,f}\big(\sfs{j}{\ext N}(s)\big)ds
    \le c h.
  \end{equation*}
  By the convergence \eqref{eq:tech-conv-to-D}, we must have \eqref{eq:tech-A-Phi-verify}.
\end{proof}

\subsection{State space collapse and asymptotic complementarity}
\label{sec:state-space-collapse}

There are two key properties leading to the proof of Theorem~\ref{thm:diffusion}. Note that the diffusion scaled stochastic processes $(\ds{\ext X},\ds{\ext W},\ds{Y},\ds{Z})$ only satisfy equations \eqref{eq:dcp-main} and \eqref{eq:dcp-mono} of the DCP problem, but do not satisfy the rest \eqref{eq:dcp-posi}, \eqref{eq:dcp-compli} and \eqref{eq:dcp-H}. We will show in the following proposition that the stochastic processes satisfy them in an approximation sense. The approximated satisfaction of \eqref{eq:dcp-posi} and \eqref{eq:dcp-H} is called \emph{state space collapse}, meaning that the workload processes are asymptotically close to the fixed point state $\invw$; The approximate satisfaction of \eqref{eq:dcp-compli} is called \emph{Asymptotic Complementarity}, and is instrumental to establish tightness.

\begin{proposition}
  \label{prop:ssc-ac}
  Pick a sample-path dependent constant $C$ such that
  \begin{equation}
    \label{eq:cond-X-osc}
    \sup_{s,t\in[0,T]}|\ds{\ext X}(t)-\ds{\ext X}(s)|\le C,
  \end{equation}
  and any $\epsilon>0$.
  Under condition \eqref{eq:initial-cond}, we have for all sufficiently large $k$
  \begin{enumerate} 
  \item State space collapse:
    \begin{equation*}
      \dist\big(\ds{\ext W}(t)\big)\le \epsilon,
      \quad \textrm{for all }t\in[0,T];
    \end{equation*}
  \item Asymptotic complementarity:
    \begin{equation*}
      \ds Y_l(t) \textrm{ can not increase at time } t
      \textrm{ if } g_l^T\ds{\ext W}(t)>2\epsilon,
      \quad \textrm{for all }t\in[0,T];
    \end{equation*}
  \item Boundedness: There exists $M>0$, depending on $C$ and network parameters, such that
    \begin{equation*}
      |\ds{\ext W}(t)|\le M,
      \quad \textrm{for all }t\in[0,T].
    \end{equation*}
  \end{enumerate}
\end{proposition}
\begin{proof}
  Due to the relationship \eqref{eq:diffusion-fluid} between the diffusion and fluid scaled processes, we just need prove these three results for the shifted fluid scaled processes, i.e.,
  \begin{align}
    \label{eq:tech-ssc}
    & \dist\big(\sfs{j}{\ext W}(s)\big)\le \epsilon,\\
    \label{eq:tech-ac}
    & \sfs{j}{Y}_l(s) = \sfs{j}{Y}_l(0) \textrm{ if }\sup_{s'\in[0,L]}g^T_l\sfs{j}{\ext W}(s')>2\epsilon,\\
    \label{eq:tech-bd}
    & |\sfs{j}{\ext W}(s)|\le M,
  \end{align}
  for all $j=0,1,\ldots, \fl{kT}$ and $s\in[0,L]$. We choose $L>T_{M,\min(\epsilon/4,\sigma/2)}+1$ with $T_{M,\min(\epsilon/4,\sigma/2)}$ specified in Theorem~\ref{thm:fluid-conv-inv}.
  We prove by induction. First, we show \eqref{eq:tech-ssc}--\eqref{eq:tech-bd} hold for $j=0$. It follows from the initial condition \eqref{eq:initial-cond}, Proposition~\ref{prop:fluid-approx} and Theorem~\ref{thm:fluid-conv-inv} that
  \begin{equation*}
    \sfs{0}{\ext W}(s) \to \chi\quad \textrm{\uoc{} on }[0,L],
  \end{equation*}
  for some $\chi\in\invw$. Though the above convergence should be interpreted as for any subsequence there is a further convergent subsequence, an easy proof by contradiction can show this is enough to prove results for all sufficiently large $k$. Thus, we omit the complication of introducing notation for subsequence. Thus \eqref{eq:tech-ssc} and \eqref{eq:tech-bd} hold for $j=0$ and all sufficiently large $k$. Moreover,
  \begin{equation*}
    |g^T_l\big(\sfs{0}{\ext W}(s)-\chi\big)|<\min(\epsilon/4,\sigma/2),
  \end{equation*}
  for all $s\in[0,L]$. This implies that
  \begin{equation}
    \label{eq:tech-ineq-1}
    \begin{split}
      &\quad|g^T_l\sfs{0}{\ext W}(s)-g^T_l\sfs{0}{\ext W}(s')| \\
      &\le |g^T_l\big(\sfs{0}{\ext W}(s)-\chi\big)| + |g^T_l\big(\sfs{0}{\ext W}(s')-\chi\big)| \\
      &\le \epsilon.
    \end{split}
  \end{equation}
  So if $\sup_{s'\in[0,L]}g^T_l\sfs{0}{\ext W}(s')>2\epsilon$ for some link $l$, then $\inf_{s'\in[0,L]}g^T_l\sfs{0}{\ext W}(s')>\epsilon$ due to the triangle inequality
  \begin{align*}
    g^T_l\sfs{0}{\ext W}(s) \ge g^T_l\sfs{0}{\ext W}(s')
    - |g^T_l\sfs{0}{\ext W}(s)-g^T_l\sfs{0}{\ext W}(s')|.
  \end{align*}
  Applying Lemma~\ref{lem:reflection-boundary}, we have
  \begin{equation}
    \label{eq:tech-Y-ac}
    \sfs{0}{Y}_l(t)-\sfs{0}{Y}_l(0) = \int_0^t\big(c_l-\ext A_l\ext\Lambda(\ds{\ext N}(s))\big)ds = 0.
  \end{equation}
  Thus \eqref{eq:tech-ac} is proved for $j=0$.

  Now assume for each $k$ there exits $j_k$ such that \eqref{eq:tech-ssc}--\eqref{eq:tech-bd} hold for all $j=0,1,\ldots, j_k-1$ for all sufficiently large $k$. Note that
  \begin{equation}
    \label{eq:overlaping}
    \sfs{j_k}{\ext W}(s) = \sfs{j_k-1}{\ext W}(1+s).
  \end{equation}
  Since $L>1$, due to overlapping, \eqref{eq:tech-ssc}--\eqref{eq:tech-bd} hold for $j=j_k$ on $[0,L-1]$. We just need to extend the result from $[0,L-1]$ to $[0,L]$. By Proposition~\ref{prop:fluid-approx} (again we omit the technicality of subsequence), as $k\to\infty$
  \begin{equation}
    \label{eq:tech-conv-j-fm}
    \sfs{j_k}{\ext W}(s)\to \bar{\ext W}(s) \quad \textrm{\uoc{} on }[0,L],
  \end{equation}
  for some fluid limit $\bar{\ext W}(\cdot)$. Due to \eqref{eq:overlaping}, we readily have $|\sfs{j_k}{\ext W}(0)|\le M$. This implies that $|\bar{\ext W}(0)|\le M$. So apply Theorem~\ref{thm:fluid-conv-inv}, we have for all $s\ge L-1\ge T_{M,\epsilon/4}$
  \begin{align}
    \label{eq:tech-ssc-j}
    & \dist\big(\bar{\ext W}(s)\big)< \min(\epsilon/4,\sigma/2),\\
    \label{eq:tech-close-j}
    & |g^T_l\big(\bar{\ext W}(s)-\chi\big)|<\min(\epsilon/4,\sigma/2),
  \end{align}
  for some $\chi\in\invw$. \eqref{eq:tech-conv-j-fm} and \eqref{eq:tech-ssc-j} imply that \eqref{eq:tech-ssc} holds for $j=j_k$ and $s\in[L-1,L]$.
  \eqref{eq:tech-conv-j-fm} and \eqref{eq:tech-close-j} imply that
  \begin{align*}
    |g^T_l\big(\sfs{j}{\ext W}(s)-\chi\big)|
    \le |g^T_l\big(\sfs{j}{\ext W}(s)-\bar{\ext W}(s)|+|g^T_l\big(\bar{\ext W}(s)-\chi\big)|
    \le \min(\epsilon/2,\sigma),
  \end{align*}
  for all $s\in[L-1,L]$. So \eqref{eq:tech-ineq-1} and \eqref{eq:tech-Y-ac} also hold for $j=j_k$ on $[L-1,L]$. By Lemma~\ref{lem:reflection-boundary}, \eqref{eq:tech-ac} is proved for $j=j_k$ and $s\in[L-1,L]$.
  The proof of boundedness \eqref{eq:tech-bd} relies on the asymptotic complementarity \eqref{eq:tech-ac}. Introduce the oscillation of a function on the interval $[a,b]$
  \begin{equation*}
    \osc{f}{[a,b]}=\sup_{a\le s\le t\le b}|f(t)-f(s)|.
  \end{equation*}
  It follows from \cite[Lemma~13]{YeYao2012} (also see \cite[Proposition~7]{KKLW2009}) that \eqref{eq:tech-ac} implies that
  \begin{align}
    \label{eq:osc-bound}
    \osc{G^T\ds{\ext W}}{[0,\frac{j_k+L}{k}]}
    & \le \kappa_c\osc{\ds{\ext X}}{[0,\frac{j_k+L}{k}]} + \kappa_c\epsilon\\
    & \le \kappa_c(C+\epsilon)\nonumber
  \end{align}
  by condition \eqref{eq:cond-X-osc}.
  Recall the definition $G= \ext A^T(\ext A \ext B \ext A^T)^{-1}$, and observe that we have
  \begin{equation}
    \label{eq:tech-matrix-AG}
    \ext A \ds{\ext W}(t) =(\ext A \ext B \ext A^T)G^T\ds{\ext W}(t).
  \end{equation}
  So there exists another constant $\kappa_a$, which only depends on $(\ext A \ext B \ext A^T)$, such that
  \begin{align*}
    |\ext A \ds{\ext W}(t)|
    \le |\ext A \ds{\ext W}(0)|+\osc{\ext A\ds{\ext W}}{[0,t]}
    \le |\ext A \chi_0| + \epsilon + \kappa_a\kappa_c(C+\epsilon)
  \end{align*}
  for all $t\le (j_k+L)/k$ and all sufficiently large $k$, where the last inequality is due to the initial condition \eqref{eq:initial-cond} and \eqref{eq:tech-matrix-AG}. Choose
  \begin{equation*}
    M = \frac{|\ext A \chi_0| + \epsilon + \kappa_a\kappa_c(C+\epsilon)}{\min_{l,r}\{A_{l,r}: A_{l,r}>0\}}.
  \end{equation*}
  Thus, $|\sfs{j_k}{\ext W}(s)|\le M$ for all $s\in[0,L]$ due to \eqref{eq:diffusion-fluid}, and \eqref{eq:tech-bd} holds for $j=j_k$.
\end{proof}

\begin{proof}[Proof of Theorem~\ref{thm:diffusion}]
  According to the functional central limit theorem (e.g., Chapter~5 of \cite{ChenYao2001}), as $k\to\infty$,
  \begin{equation}
    \label{eq:conv-primitive}
    \ds{\ext E}(t) \dto \hat{\ext E}(t)
    \ \text{ and } \
    \ds{\ext S}(t) \dto \hat{\ext S}(t),
  \end{equation}
  where $\hat{\ext E}_{r,f}(t)$ and $\hat{\ext S}_{r,f}(t)$ are standard Brownian motions independent of each other.
  Using the Skorohod representation theorem, we can map all random objects to the same probability space on which the above convergence, as well as the convergence \eqref{eq:initial-cond}, holds  a.s. So we employ sample-path arguments for the rest of this proof.

  We first show the convergence of $\ds{\ext X}(t)$.
  Consider the fluid scaled process by factor $k^2$ instead of $k$, and define $\tilde{\ext W}^n(t):=\ext W^k(k^2t)/k^2$. The fluid approximation result, Proposition~\ref{prop:fluid-approx}, still holds. Note that by condition \eqref{eq:initial-cond}, as $k\to\infty$,
  \begin{equation*}
    \tilde{\ext W}^k(t) = \frac{1}{k}\ds{\ext W}(0)\to 0\in\invw.
  \end{equation*}
  This implies, by Theorem~\ref{thm:fluid-conv-inv}, that, as $k\to\infty$,
  \begin{equation}
    \label{eq:D-cov-rho}
    \fss{\ext D}_{r,f}(t) \to \ext\rho_{r,f}t,\quad\textrm{\uoc{} on }[0,\infty).
  \end{equation}
  The convergence \eqref{eq:D-cov-rho}, together with \eqref{eq:conv-primitive} (almost sure convergence version), implies that
  \begin{equation}
    \label{eq:D-cov-S}
    \ds{\ext S}_{r,f,f'}\big(\tilde{\ext D}_{r,f}(t)\big) \to \hat{\ext S}_{r,f,f'}(\ext\rho_{r,f}t),
    \quad\textrm{\uoc{} on }[0,\infty).
  \end{equation}
  Let
  \begin{equation*}
    \hat{\ext U}(t) = \sum_{f'\in\phase_r}\hat{\ext S}_{r,f',f}(\ext\rho_{r,f'}t)
    -\sum_{f'\in\phase_r \cup \{0\}}\hat{\ext S}_{r,f,f'}(\ext\rho_{r,f}t).
  \end{equation*}
  Recall \eqref{eq:X-center-scaled}, the diffusion scaled version of the system dynamics \eqref{eq:centered-k}.
  From the above convergence \eqref{eq:conv-primitive}--\eqref{eq:D-cov-S}, we can conclude that, \uoc{} on $[0,\infty)$,
  \begin{equation}
    \label{eq:X-conv}
    \ds{\ext X}(t) \to \hat{\ext X}(t),
  \end{equation}
  where $\hat{\ext X}(t) = - \ext\theta t + \diag{\ext m} (I-\ext P^T)^{-1} (\hat{\ext E}(\ext a t) + \hat{\ext U}(t))$.
  Clearly, it has drift $-\ext\theta$.
  We now show that the covariance matrix is \eqref{eq:COV}.
  The covariance matrix of $\hat{\ext E}(\ext a t)$ is $\diag {\lambda \ext a}$.
  To compute the covariance matrix of $\hat{\ext U}(t)$, we only need to do that for each fixed $r\in\route$.
  Note that each $\hat{\ext S}^k_{r,f,f'}(\ext\rho_{r,f}t)$, $f\in\phase_r$, $f'\in\phase_r\cup\{0\}$, is an independent Brownian motion with variance $\ext{\rho}_{r,f}\ext{\mu}_{r,f} P^r_{f,f'}$.
  Observe that
  \begin{align*}
    &\quad \E[\hat{\ext U}_{f_0}(t)\hat{\ext U}_{f_1}(t)] \\
    &= \E\Bigg[
      \Big(  \sum_{g\in\phase_r} \hat{\ext S}_{r,g,{f_0}}(\ext\rho_{r,g}t)
           - \sum_{g\in\phase_r\cup\{0\}} \hat{\ext S}_{r,{f_0},g}(\ext\rho_{r,f_0}t)  \Big)
      \Big(  \sum_{g\in\phase_r} \hat{\ext S}_{r,g,{f_1}}(\ext\rho_{r,g}t)
           - \sum_{g\in\phase_r\cup\{0\}} \hat{\ext S}_{r,{f_1},g}(\ext\rho_{r,f_1}t)  \Big)
    \Bigg]
  \end{align*}
  Writing out this product we get an expression of the form $I-II-III+IV$. We compute each term separately. Let $\id{\cdot}$ be the indicator function.
  \begin{align*}
    I  &= \id{f_0=f_1} \sum_{g\in\phase_r} \E\Big[\hat{\ext S}_{r,g,f_0}^2(\ext\rho_{r,g}t)\Big]
       = \id{f_0=f_1} \sum_{g\in\phase_r} \ext{\mu}_{r,g} \ext\rho_g P_{gf_0},\\
    IV &= \id{f_0=f_1} \ext{\mu}_{r,f_0}\ext{\rho}_{r,f_0}, \hspace{1cm}
    II = \ext\rho_{r,f_1}\ext{\mu}_{r,f_1} P^r_{f_1,f_0},\hspace{1cm}
    III = \ext\rho_{r,f_0}\ext{\mu}_{r,f_0} P^r_{f_0,f_1}.
  \end{align*}
  Thus the covariance matrix of $\hat{\ext U}$ is given by \eqref{eq:COV-S}, from which we obtain \eqref{eq:COV}.

  Second, we study the convergence of $\ds Z(t)$. By Proposition~\ref{prop:ssc-ac}~(a), as $k\to\infty$,
  \begin{equation*}
    |h^T_m\ds{\ext W}(t)|\to 0, \quad\textrm{\uoc{} on }[0,\infty).
  \end{equation*}
  Multiplying both side of the diffusion scaled version of \eqref{eq:dynamics-k-workload-decomp}, we have
  \begin{equation*}
    h^T_m\ds{\ext W}(t) = h^T_m\ds{\ext W}(0) + h^T_m\ds{\ext X}(t) + \ds Z(t).
  \end{equation*}
  So as $k\to\infty$,
  \begin{equation}
    \label{eq:Z-conv}
    \ds{Z}(t)\to \hat Z(t):=-h^T_m\hat{\ext X}(t).
  \end{equation}

  Next, we study the convergence of $\ds{Y}$. It follows from Proposition~\ref{prop:ssc-ac}~(c) that $\ds Y(t)$ is also uniformly bounded on the interval $[0,T]$. Hence, according to Helly's selection theorem (e.g., \cite[p.\ 336]{Billingsley1995}), for any subsequence of $\ds Y(t)$, there exists a further subsequence $\mathcal K$ along which as $k\to\infty$
  \begin{equation}
    \label{eq:Y-conv}
    \ds Y(t)\to\hat Y(t),
  \end{equation}
  for non-decreasing function $\hat Y(t)$ which are continuous almost everywhere. The above convergence hold for all time $t\in[0,T]$ at which $\hat Y(t)$ is continuous.

  Summarizing \eqref{eq:X-conv}--\eqref{eq:Y-conv}, by \eqref{eq:dynamics-k-workload-decomp}, we have along the subsequence $\mathcal K$ as $k\to\infty$,
  \begin{equation*}
    \ds{\ext W}(t) \to \hat{\ext W}(t)
    = \hat{\ext W}(0) + \hat{\ext X}(t) + \ext B G \hat Y(t) + \ext B H \hat Z(t)
  \end{equation*}
  for almost all $t\in[0,L]$ (those $t$ at which $\hat Y(t)$ is continuous). Note that $\hat Y(t)$ can be chosen to be right continuous with left limit since it is continuous almost everywhere. Thus, $\hat{\ext W}(t)$ is also right continuous with left limit. By Proposition~\ref{prop:ssc-ac},  the $\big(\ds{\ext W},\ds{\ext X},\ds Y,\ds Z\big)$ satisfies the DCP \eqref{eq:dcp-main}--\eqref{eq:dcp-H}. It follows from the oscillation bound \eqref{eq:osc-bound} that the limit $\hat{\ext W}(t)$ is continuous, and so is the process $\hat Y(t)$. By the uniqueness of the solution to the DCP problem (e.g., \cite[Proposition~4]{YeYao2012}), the convergence along the subsequence $\mathcal K$ implies the convergence along the original sequence.
\end{proof}

\section{The invariant distribution: insensitivity and product form}
\label{sec:invariant}

In this section we analyze the SRBM $\hat{\ext N}(t)$; the limit of our queue length process.
Define
\begin{equation*}
  \hat W_G(t) = G^T \hat{\ext W}(t).
\end{equation*}
It follows from \eqref{eq:ext-B-equiv} and \eqref{eq:G-matrix} (in particular $G^T \ext B \ext A^T = I$) that
\begin{equation*}
  \hat{\ext W}(t) = \ext B\ext A^T \hat W_G(t) = \ext B^{\dagger} \ext A^T \hat W_G(t).
\end{equation*}
By \eqref{eq:workload-phase}, and $1/\ext m$ the vector with each component being the reciprocal of the corresponding one of $\ext m$,
\begin{align*}
  \hat{\ext N}(t) &= (I-\ext P^T) \diag {1/\ext m} \hat{\ext W}(t).
\end{align*}
According to the definition of $\ext B^\dagger$,
\begin{equation*}
  \hat{\ext N}(t) = \diag {\ext \rho} \ext A^T \hat W_G(t).
\end{equation*}
Recall that \eqref{eq:A-ext-A}.
Since $N_r(t)=\sum_{f\in\phase_r}\ext N_{r,f}(t)$ and same relation holds for the diffusion limits, the limiting queue length process at the route level satisfies
\begin{equation*}
  \hat N(t) = \diag \rho A^T \hat W_G(t).
\end{equation*}
We derive the invariant distribution for $\hat W_G(t)$:
\begin{theorem}
  \label{thm:productform}
  Assume $\theta>0$.
  As $t\rightarrow\infty$, $\hat W_G(t)\rightarrow \hat W_G(\infty)$ in distribution, where the random variable $\hat W_G(\infty)$ is a vector of independent exponential distributions with rate  $\theta$.
\end{theorem}
We prove this theorem by checking a condition for product form, due to \cite{HarrisonWilliams1987}. A version of this result suitable for our purposes is stated in Section \ref{ss-skew}. The condition involves a relationship between the covariance matrix and the reflection matrix which are analyzed in Section \ref{ss-cov} and \ref{ss-rev}. All insights are combined in Section \ref{ss-prod}.

\subsection{Sufficient condition for product form}
\label{ss-skew}

A SRBM is characterized by the drift $-\theta$, covariance matrix $\Gamma$ of the free process, and reflection matrix $R$. The SRBM has a stationary distribution as we assume $\theta>0$. \cite{HarrisonWilliams1987} have shown when this stationary distribution is of product form assuming a normalized form of $R$. For our purposes the version presented as Theorem 7.12 in \cite{ChenYao2001} is most convenient, and we follow that verbatim here.
Suppose that $R^{-1} \theta>0$.  Let $\Gamma_d$ be a diagonal matrix containing the diagonal elements of $\Gamma$, and let $R_d$ be a diagonal matrix containing the diagonal elements of $R$.
If
\begin{equation}
  \label{eq:productformcondition}
  2\Gamma = RR_d^{-1}\Gamma_d + \Gamma_d R_d^{-1} R^T,
\end{equation}
the density of the stationary distribution is given by
\begin{equation*}
  f(z) = \prod_{r\in\route} \sigma_r e^{-\sigma_r z},
\end{equation*}
where the $R$-dimensional vector $\sigma = 2\Gamma_d^{-1} R_d \theta$.
We need to verify this in our situation.
From the discussion of the reflection mapping in Section~\ref{sec:geometry-fixed-point}, in particular \eqref{eq:dcp-main-alt}, we have
\begin{align*}
  \hat W_G(t) &= \hat W_G(0) + G^T\hat{\ext X}(t) + G^T\ext BG Y(t) \nonumber\\
  &= \hat W_G(0) +  (\ext A\ext B\ext A^T)^{-1}\ext A\hat{\ext X}(t) + (\ext A\ext B\ext A^T)^{-1}Y(t),
\end{align*}
where the last inequality follows from the definition of $G$ (recall $G =\ext A^T (\ext A\ext B\ext A^T)^{-1}$) and \eqref{eq:G-matrix}.
So the reflection matrix $R= (\ext A\ext B\ext A^T)^{-1}$.
Since in our case the reflection matrix is symmetric, the sufficient condition \eqref{eq:productformcondition} becomes
\begin{equation}
  \label{eq:productformcondition-sym}
  \Gamma = R R_d^{-1} \Gamma_d.
\end{equation}
In Section~\ref{ss-cov}, we derive an expression for the covariance matrix of $A\hat{\ext X}(t)$. Then in Section~\ref{ss-rev}, we simplify the reflection matrix $R$.
Together, they also yield the covariance matrix of $G^T\hat{\ext X}(t) = R\ext A\hat{\ext X}(t)$.
We verify \eqref{eq:productformcondition-sym} in Section~\ref{ss-prod}.

\subsection{The covariance matrix}
\label{ss-cov}

The covariance matrix of $\ext A \ext X$ is
\begin{equation}
  \label{eq:COV-AX}
  \ext A\Sigma_X \ext A^T = AC \Sigma_X C^T A^T,
\end{equation}
by Theorem~\ref{thm:diffusion} and \eqref{eq:A-ext-A}.
In view of \eqref{eq:COV}, to simplify the notation, let $\vect{\tau}_r = (I-P^r)^{-1} \ext{m}_r$. In other words, $\vect{\tau}_r=(\vect{\tau}_{r,1},\ldots, \vect{\tau}_{r,F_r})^T$ where $\vect{\tau}_{r,f}$ can be interpreted as the residual service time of a job at phase $f$ on route $r$.
Note that by \eqref{eq:COV} and \eqref{eq:COV-S}, $\Sigma_X$ is a block diagonal matrix with $r$th block being an $F_r$-dimensional matrix
\begin{equation*}
  \Sigma^r_X = \lambda_r\diag{\vect{a}_r}+\Sigma^r_U,
\end{equation*}
where
\begin{equation}
  \label{eq:COV-S_r}
  \Sigma_U^r = \diag{ (I + P^{r,T}) (\vect{\rho}_r\cdot \vect{\mu}_r)} -P^{r,T} \diag {\vect{\rho}_r\cdot \vect{\mu}_r} - \diag {\vect{\rho}_r \cdot \vect{\mu}_r} P^r.
\end{equation}
Due to the structure of $\ext C$ (c.f.\ Section~\ref{sec:geometry-fixed-point}), the matrix $\ext C\Sigma_X \ext C^T$ is an $R\times R$ diagonal matrix, with on each diagonal entry an expression of the form
\begin{equation}
  \label{eq:cov-r-th}
  \vect{\tau}_r^T \left( \lambda_r \diag{\vect{a}_r} + \Sigma_U^r\right)\vect{\tau}_r.
\end{equation}
To compute the above value, we first need to simplify $\Sigma_U^r$.
Note that, by \eqref{eq:rho-phase}
\begin{equation}
  \label{eq:rho-vec-r}
  \vect{\rho}_r = \lambda_r \diag{\ext{m}_r} (I- P^{r,T})^{-1} \vect{a}_r,
\end{equation}
we see that
\begin{equation*}
  \vect{\rho}_r \cdot \vect{\mu}_r = \lambda_r (I-P^{r,T})^{-1} \vect{a}_r.
\end{equation*}
Thus, we have
\begin{align*}
  (I + P^{r,T}) (\vect{\rho}_r\cdot \vect{\mu}_r) = \lambda_r (I + P^{r,T})(I-P^{r,T})^{-1}\vect{a}_r = \lambda_r [2(I-P^{r,T})^{-1}-I] \vect{a}_r.
\end{align*}
So the first term on the right hand side of \eqref{eq:COV-S_r} can be transformed into $2 \lambda_r\diag{(I-P^{r,T})^{-1}\vect{a}_r}-\lambda_r\diag{\vect{a}_r}$.
The second and the third terms on the right hand side of \eqref{eq:COV-S_r} are just transpose of each other, thus they play the same role in computing the quadratic form \eqref{eq:cov-r-th}.
This implies that \eqref{eq:cov-r-th} can be written as
\begin{align}
  \quad 2\lambda_r
  \vect{\tau}_r^T
  \left(
    \diag{\vect{a}_r^T (I-P^r)^{-1}} (I-P^r)
  \right)
  \vect{\tau}_r
  &= 2 \lambda_r
  \ext{m}_r^T (I-P^{r,T})^{-1} \diag{\vect{a}_r^T (I-P^r)^{-1} } \ext{m}_r \nonumber\\
  &= 2 \lambda_r
  \ext{m}_r^T (I-P^{r,T})^{-1} \vect{\rho}_r,
  \label{eq:r-th-interm}
\end{align}
where the first equality is due to the definition of $\vect{\tau}_r$ in the above.
Let $\beta^{(2)}_r$ be the second moment of the phase-type distribution specified by $\vect{a}_r$ and $P^r$.
We now show that \eqref{eq:r-th-interm} equals $\lambda_r \beta_r^{(2)}$.
The normalized load vector $\vect{\rho}_r/\rho_r$ has a renewal-theoretic interpretation: for a renewal process with phase-type inter-renewal times $\ext\rho_{r,f}/\rho_r$ contains the probability that the renewal process is in phase $f$ in stationarity.
Using renewal theory, and recalling \eqref{eq:mean-MixErlang}, we see that
\begin{align*}
  \frac{\beta_r^{(2)}}{2\beta_r}
  = \frac{\vect{\tau}_r^T\vect{\rho}_r}{\rho_r}
  =  \frac  {\ext{m}_r^T (I-P^{r,T})^{-1} \diag{\ext{m}_r} (I-P^{r,T})^{-1} \vect{a}_r }{\beta_r}
  =  \frac  {\ext{m}_r^T (I-P^{r,T})^{-1} \vect{\rho}_r }{\beta_r},
\end{align*}
where the last equality is due to \eqref{eq:rho-vec-r}.
Consequently,
\begin{equation*}
  \beta^{(2)}_r=2 \ext{m}_r^T (I-P^{r,T})^{-1} \vect{\rho}_r.
\end{equation*}
In view of \eqref{eq:cov-r-th}--\eqref{eq:r-th-interm},  the $r$th element of the diagonal matrix $C \Sigma_X C^T$ is $\lambda_r \beta^{(2)}_r$.
Thus, setting $\beta^{(2)}=(\beta_1^{(2)},\ldots, \beta_R^{(2)})$,
\begin{equation*}
C \Sigma_X C^T = \diag{\lambda \cdot \beta^{(2)}}.
\end{equation*}
By \eqref{eq:COV-AX}, we conclude that the covariance matrix of $\ext A\ext X(t)$ is $A \diag{\lambda \cdot \beta^{(2)}} A^T$.

\subsection{The reflection matrix}
\label{ss-rev}

By \eqref{eq:A-ext-A}, the reflection mapping can be written as
\begin{equation*}
  R = (\ext A \ext B \ext A^T)^{-1} = (A \ext C \ext B \ext C^T A^T)^{-1}.
\end{equation*}
According to \eqref{eq:B}, $\ext B$ is a $\sum_{r\in\route}F_r$-dimensional diagonal matrix.
Due to the structure of $\ext C$ (see Section~\ref{sec:geometry-fixed-point}), $\ext C \ext B \ext C^T$ is a $R$-dimensional diagonal matrix, with the $r$th element being the sum of the all the elements on the diagonal of the $r$th block of $B$.
Thus, by \eqref{eq:rho-phase}, the $r$th diagonal element of  $\ext C \ext B \ext C^T$ is
\begin{align*}
  \lambda_r \ext{m}_r^T(I-P^{r,T})^{-1}\diag{\ext{m}_r}(I-P^{r,T})^{-1}\vect{a}_r
  = \lambda_r \ext{m}_r^T(I-P^{r,T})^{-1}\vect{\rho}_r
  = \lambda_r\beta^{(2)}_r/2.
\end{align*}
So we have $R = \frac 12 (A \diag{\lambda \cdot \beta^{(2)}} A^T )^{-1}$.

\subsection{Verification of skew symmetry condition}
\label{ss-prod}

We are now in a position to verify the product form condition \eqref{eq:productformcondition-sym}.

\begin{proof}[Proof of Theorem~\ref{thm:productform}]
Set $D = \diag\lambda \diag { \beta^{(2)}}$. In the previous two sections we derived for the reflection matrix
 $R = (ADA^T)^{-1}/2$ and for the covariance matrix $\Sigma = G^T A\Sigma_XA^T G$, which equals  $R A D A^T R$. This implies that $\Sigma =2R$.
 The product form condition
\eqref{eq:productformcondition-sym}  which is $R^{-1} \Sigma = \Sigma_d R_d^{-1}$ is equivalent to $R^{-1} \Sigma = \Sigma_d R_d^{-1}$ which is now trivial: both sides  equal $2$.
The vector $\sigma = 2\Gamma_d^{-1} R_d \theta = \theta$; see also  \cite{HarrisonWilliams1987} and \cite{ChenYao2001}.
\end{proof}

\section*{Acknowledgments}
This research is made possible by grants from the `Joint Research Scheme' program, sponsored by the Netherlands Organization of Scientific Research (NWO) and the Research Grants Council of Hong Kong (RGC) through projects 649.000.005 and D-HK007/11T, respectively. MV is also affiliated with CWI, and is supported by a MEERVOUD grant from Netherlands Organisation for Scientific Research (NWO). BZ is also affiliated with VU University, Eindhoven of Technology, and Georgia Institute of Technology, and is supported by an NWO VIDI grant and an IBM faculty award.

\bibliography{pub}

\end{document}